\documentclass[preprint,cleveref]{colt2026} 

\title[Network approximation of smooth functions on sets with low intrinsic dimension]{Optimal neural network approximation of smooth compositional functions on sets with low intrinsic dimension}
\usepackage{times}

\usepackage{graphicx} 
\usepackage{amsmath} 
\usepackage{amssymb} 
\usepackage{bm}
\usepackage{enumerate}
\usepackage{hyperref}
\usepackage{booktabs}
\usepackage{appendix}
 
\hypersetup{
	colorlinks   = true,
	urlcolor     = blue,
	linkcolor    = blue,
	citecolor   = blue
}
\usepackage{aliascnt}

\usepackage{bm}
\usepackage{dsfont}

\newcommand{\ind}{\mathds{1}}
\newcommand{\R}{\mathds{R}}
\newcommand{\N}{\mathds{N}}

\providecommand{\P}{}
\renewcommand{\P}{\mathds{P}}

\newcommand{\E}{\mathds{E}}

\newcommand{\be}{\bm{e}}

\newcommand{\bu}{\bm{u}}
\newcommand{\bv}{\bm{v}}

\newcommand{\bx}{\bm{x}}
\newcommand{\by}{\bm{y}}
\newcommand{\bz}{\bm{z}}

\newcommand{\bX}{\bm{X}}

\newcommand{\Bcal}{\mathcal{B}}

\newcommand{\Fcal}{\mathcal{F}}
\newcommand{\Gcal}{\mathcal{G}}
\newcommand{\Hcal}{\mathcal{H}}

\newcommand{\Mcal}{\mathcal{M}}
\newcommand{\Ncal}{\mathcal{N}}

\newcommand{\Pcal}{\mathcal{P}}

\newcommand{\Scal}{\mathcal{S}}

\newcommand{\eps}{\varepsilon}

\newcommand{\balpha}{\bm{\alpha}}
\newcommand{\bbeta}{\bm{\beta}}

\renewcommand{\vec}{\operatorname{vec}}

\renewcommand{\bar}{\overline}

\DeclareMathOperator*{\argmin}{arg\,min}

\newcommand{\wh}[1]{\widehat{#1}}
\newcommand{\wt}[1]{\widetilde{#1}}

\makeatletter
\def\blfootnote{\gdef\@thefnmark{}\@footnotetext}
\makeatother

\numberwithin{theorem}{section}
\numberwithin{lemma}{section}
\numberwithin{proposition}{section}
\numberwithin{corollary}{section}
\numberwithin{assumption}{section}
\numberwithin{remark}{section}
\numberwithin{example}{section}
\numberwithin{definition}{section}

\Crefname{theorem}{Theorem}{Theorems}
\Crefname{corollary}{Corollary}{Corollaries}
\Crefname{lemma}{Lemma}{Lemmas}
\Crefname{proposition}{Proposition}{Propositions}
\Crefname{assumption}{Assumption}{Assumptions}
\Crefname{definition}{Definition}{Definitions}
\Crefname{remark}{Remark}{Remarks}
\Crefname{example}{Example}{Examples}

\newcommand{\bit}{\mathrm{bit}}

\coltauthor{%
 \Name{Thomas Nagler} \Email{t.nagler@lmu.de}\\
 \addr Department of Statistics, LMU Munich \\
\addr Munich Center for Machine Learning%
 \AND
 \Name{Sophie Langer} \Email{s.langer@rub.de}\\
 \addr Faculty of Mathematics, Ruhr-Universit\"at Bochum%
}

\begin{document}

\maketitle

\begin{abstract}%
    We study approximation and statistical learning properties of deep ReLU networks under structural assumptions that mitigate the curse of dimensionality. We prove minimax-optimal uniform approximation rates for $s$-H\"older smooth functions defined on sets with low Minkowski dimension using fully connected networks with flexible width and depth, improving existing results by logarithmic factors even in classical full-dimensional settings. A key technical ingredient is a new memorization result for deep ReLU networks that enables efficient point fitting with dense architectures.
    We further introduce a class of compositional models in which each component function is smooth and acts on a domain of low intrinsic dimension. This framework unifies two common assumptions in the statistical learning literature, structural constraints on the target function and low dimensionality of the covariates, within a single model. We show that deep networks can approximate such functions at rates determined by the most difficult function in the composition. As an application, we derive improved convergence rates for empirical risk minimization in nonparametric regression that adapt to smoothness, compositional structure, and intrinsic dimensionality.
\end{abstract}


\section{Introduction}

The empirical success of deep neural networks in high-dimensional learning problems suggests that they can efficiently adapt to the low-complexity structure of the target function. Understanding which structural assumptions enable such behavior and how they can be exploited by concrete network architectures remains a central question in approximation theory and statistical learning. 

A classical structural assumption is the smoothness of the target function. Several authors have proved sharp approximation and estimation results for ReLU networks when the target function is $s$-H\"older smooth \citep{yarotsky2017error,Lu2021,shen2022optimal,ou2024covering}. However, these results inevitably suffer from the curse of dimensionality: approximation and estimation rates scale exponentially in the input dimension $D$. In mean regression problems, this leads to minimax rates of order $n^{-2s/(2s+D)}$ \citep{stone1982optimal}, which quickly become impractical even for moderate $D$. Thus, smoothness alone is insufficient to explain the effectiveness of neural networks in high-dimensional settings.
One prominent approach to overcoming the curse of dimensionality is to assume some compositional structure of the target function. For such models, deep ReLU networks can achieve approximation and estimation rates that depend on the smoothness and dimensionality of the component functions rather than on the ambient dimension, thereby mitigating the curse of dimensionality \citep{schmidt2020nonparametric,kohler2021rate,danhofer2025position}.

An alternative line of work focuses on the intrinsic dimensionality of the input space. Motivated by the manifold hypothesis, this literature assumes that although the ambient dimension may be large, the data lie on a set of much lower intrinsic dimension. 
Under such assumptions, approximation and learning rates depend on the intrinsic rather than the ambient dimension \citep{bickel2007local, kpotufe2011knn}. Several works show that ReLU networks can efficiently approximate smooth functions restricted to low-dimensional manifolds or, more generally, sets with bounded Minkowski dimension \citep{nakada2020adaptive,schmidthieber2019deeprelunetworkapproximation, kohler2023estimation}. However, these results require very wide and relatively shallow architectures, with some imposing additional sparsity on the network parameters. This leaves an unsatisfying gap between existing theory and the architectures used in practice.
\citet{jiao2023deep} provide similar results for dense, deep networks, but their approximation and estimation rates are suboptimal by polynomial factors in network size.

From both modeling and theoretical perspectives, compositional structure and low intrinsic dimensionality should not be viewed as competing explanations for the success of neural networks. \citet{schulte25a} impose a combination of both assumptions in a causal inference setting, but measure intrinsic dimensionality only in the input layer, not in the intermediate layers of the composition.
This can be inefficient, as component functions often also act on domains of low intrinsic dimension. For instance, \citet{ansuini2019intrinsic} showed that trained deep networks typically have hidden-layer representations concentrated on low-dimensional sets, and exploiting this property is crucial for obtaining sharp guarantees.


The main contributions of this paper are as follows.
\begin{itemize}
\item \textbf{Optimal approximation on low-dimensional sets.} We prove a new approximation result for fully connected ReLU networks for $s$-H\"older smooth functions defined on sets with low Minkowski dimension (\Cref{thm:main-approx}). The result allows for arbitrary width-to-depth ratios and uses a network size that is minimax-optimal (up to constants) for this problem (\Cref{prop:optimality-bounded}). Even when intrinsic dimension equals ambient dimension, \Cref{thm:main-approx} improves on the best known result for $(s > 1)$-smooth functions by logarithmic factors in both, width and depth.

\item \textbf{A new memorization result for deep ReLU networks.} We provide a new memorization result for deep ReLU networks with flexible width and depth (\Cref{prop:point-fitting-ours}), which may be of independent interest. The result improves upon existing results of \citet{vardi2022on} in terms of network size and forms a key technical ingredient of our approximation results.

\item \textbf{Approximation of compositional models with low-dimensional components.} We propose a new model class of compositional functions, where each component is smooth and is defined on a set with low Minkowski dimension. This class combines two common assumptions imposed in the deep learning literature, i.e., the compositional structure of the target function and the low intrinsic dimension of its inputs. In \Cref{thm:main-compositional}, we prove that deep ReLU networks can approximate functions from this class at rates determined by the intrinsic dimensionality and smoothness of the component functions.

\item \textbf{Statistical rates for nonparametric regression.} In \Cref{sec:stat}, we apply the approximation results to nonparametric regression and derive improved convergence rates for empirical risk minimization with neural networks. The resulting statistical rates adapt simultaneously to smoothness, compositional structure, and intrinsic dimensionality and improve upon existing results for deep fully connected networks.
\end{itemize}

\section{Preliminaries}
In our analysis, we consider vanilla feedforward neural networks with ReLU activation function defined as follows.
\begin{definition}[Class of fully connected neural networks]
    Let $\sigma : \R \to \R$, $\sigma(x)=\max\{x,0\}=(x)_+$ be the ReLU activation function, and for $v \in \R$ denote by
    $\sigma_v(t) = \sigma(t + v)$ the $v$-biased ReLU function.
    For vectors $\bx = (x_1, \dots, x_D)^\top \in \R^D$ and $\bv = (v_1, \dots, v_D)^\top \in \R^D$, we apply $\sigma_{\bv}$ componentwise, i.e.,
    \begin{align*}
        \sigma_{\bv}(\bx) = (\sigma(x_1+v_1), \dots, \sigma(x_D+v_D))^\top = ((x_1+v_1)_+, \dots, (x_D+v_D)_+)^\top.
    \end{align*}
    A function $f : \R^{N_0} \to \R^{N_{L + 1}}$ is called a
    \emph{fully connected neural network} with depth $L \in \N$
    and width vector $(N_0, N_1, \dots, N_{L + 1})$
    if it can be written as
    \begin{align}
        \label{eq:FNN}
        f(\bx) = \bv_L +  W_L  \sigma_{\bv_{L-1}} \cdots W_1 \sigma_{\bv_0} W_0 \bx ,
    \end{align}
    where for each $i = 0, \dots, L$,
    $W_i \in \R^{N_{i + 1} \times N_{i}}$ is the weight matrix and
    $\bv_i \in \R^{N_{i + 1}}$ is the bias vector. Let $B_i = \max\{\max_{j,k} |W_{i, j,k}|, \|\bv_i\|_{\infty}\}$ be the maximum parameter magnitude in layer $i$.
    The class of such networks with depth $L$,
    maximum width $N \in \N$, and maximum parameter magnitude $B \in \R$
    is defined as
    \begin{align*}
        \mathcal{F}_{d_0, d_{L + 1}}(N, L, B)
        = \Bigl\{ f\text{ of the form \eqref{eq:FNN}}
        : N_0 = d_0,  N_{L + 1} = d_{L + 1}, \max_{i} N_i \le N, \max_{i} B_i \le B\Bigr\}.
    \end{align*}
\end{definition}
To derive non-trivial results, analyses of approximation rates and statistical risk bounds require structural assumptions on the target function \citep[see, e.g.,][Section 3]{gyorfi2002distribution}. In this work, we consider the class of $s$-Hölder smooth functions defined as follows.
\begin{definition}[Class of Hölder smooth functions]
    Let \( K \subset \R^D \), \(s > 0 \), and \( C > 0 \) be constants.
    For a multiindex \( \bm{\alpha} = (\alpha_1, \ldots, \alpha_D) \in \N_0^D \),
    define its order as \( |\bm{\alpha}| = \alpha_1 + \cdots + \alpha_D \), and
    let
    $\partial^{\bm{\alpha}} f =
        {\partial^{|\bm{\alpha}|} f}/
        {\partial x_1^{\alpha_1} \cdots \partial x_D^{\alpha_D}}.$
    Denote by \( \lfloor s \rfloor \) the integer part of \(s \).
    Then, the Hölder class of functions \( \mathcal{H}_{s}(K, C) \) is defined as
    \begin{equation}
        \begin{aligned}
            \mathcal{H}_{s}(K,C)
            = \Bigg\{ f\colon K \to \R
             & \sum_{\substack{\bm{\alpha} : \\ |\bm{\alpha}| < \lfloor s \rfloor}}
            \|\partial^{\bm{\alpha}} f\|_{\infty}  +
            \sum_{\substack{\bm{\alpha} :    \\ |\bm{\alpha}| = \lfloor s \rfloor}}
            \sup_{\substack{\bx, \by \in K   \\ \bx \neq \by}}
            \frac{|\partial^{\bm{\alpha}} f(\bx) - \partial^{\bm{\alpha}} f(\by)|}
            {\|\bx - \by\|_{\infty}^{s - \lfloor s \rfloor}}
            \le C
            \Bigg\}.
        \end{aligned}
    \end{equation}
\end{definition}
The intrinsic dimensionality of data is measured via its Minkowski dimension, defined as follows.

\begin{definition}[Covering number]
    Let $(V, \|\cdot\|)$ be a normed space and $\Theta \subset V$. $\{V_1, \dots, V_N\}$ is an $\epsilon$-cover of $\Theta$ if for each $t \in \Theta$ there exists an $i \in \{1, \dots, N\}$ and $t' \in V_i$ such that $\|t-t'\| \leq \epsilon$. The $\epsilon$-covering number is the minimal size of an $\epsilon$-cover, i.e., 
    \begin{align*}
        \mathcal{N}(\epsilon, V, \|\cdot\|) := \min\{N: \exists \epsilon\text{-cover over } \Theta \text{ of size } N\}.
    \end{align*}
\end{definition}
\begin{definition}[Minkowski dimension]
    The Minkowski dimension of a bounded set $\Mcal \subset \R^D$ is 
    \begin{align*}
        \dim_M(\Mcal) := \limsup_{\epsilon \to 0} \frac{\log \Ncal(\epsilon, \Mcal, \| \cdot \|_\infty)}{\log(1/\epsilon)}.
    \end{align*}
\end{definition}
Note that the Minkowski dimension describes a broader class of low-dimensional sets compared to other dimensionalities. 
For instance, the Minkowski dimension equals the (intrinsic) manifold dimension for smooth manifolds $\Mcal$, but it also applies to less regular sets such as fractals (see \cite{nakada2020adaptive} for more details).
\\
\\
\textbf{Notations.} Vectors are denoted in bold letters, i.e., $\bx=(x_1, \dots, x_D)^T$. We write $a \lesssim b$ if there exists a constant $C \in (0,\infty)$ such that $a \leq C b$, and $a \asymp b$ if $a \lesssim b$ and $b \lesssim a$. We write $\lfloor z \rfloor$ for the largest integer less than or equal to $z \in \R$.

\section{Approximation of smooth functions on sets with low Minkowski dimension}
\label{sec:manifold}

We start with an approximation result for ReLU networks for Hölder classes of functions having domains with low intrinsic dimension.
The numerical constants in such results typically depend exponentially on dimension and smoothness.
To better highlight the key mathematical ideas and to simplify notation, we do not track the precise values of these constants in the following.
In particular, we write
\begin{align*}
    \phi \in \wt  \Fcal(N, L, B) \quad \text{whenever} \quad \exists\, C = C(D, s) \in (0, \infty) \text{ such that } \phi \in \Fcal(CN, CL, B^C).
\end{align*}

\subsection{Main result}
The following theorem shows that networks with variable width and depth can efficiently approximate $s$-Hölder functions on sets with low intrinsic dimensionality. 
\begin{theorem} \label{thm:main-approx}
    Let $\eps > 0$, $S = [a_j, b_j]_{j = 1}^D$, $f \in \Hcal_s(S, C)$, $\Mcal \subseteq S$ be a set with covering number $\mathcal{N}(\eps^{1/s}, \Mcal, \|\cdot\|_{\infty}) \lesssim \eps^{-d/s}$. Then for $N, L \in \N$ with $N^2L^2 \asymp \eps^{-d/s}$ there is a neural network
    \begin{align*}
        \phi \in \wt \Fcal_{D, 1}(N, L + L \sqrt{\log L / \log N}, N) \quad \text{such that} \quad  \sup_{\bx \in \Mcal} |f(\bx) - \phi(\bx)| \lesssim \eps.
    \end{align*}
    In particular, if $\dim_M(\Mcal) = d$, it holds 
        $\sup_{\bx \in \Mcal} |f(\bx) - \phi(\bx)| \lesssim (NL)^{-2s/d}.$
\end{theorem}

\begin{table}[t]
    \centering
    \small
    \setlength{\tabcolsep}{5pt}
    \begin{tabular}{lllll}
        \toprule
        Reference                                              &
        Structure of $\mathcal M$                              &
        Network                                                &
        Width                                                  & Depth                                                                 \\
        \midrule
        This work                                    & Approx.~Minkowski-$d$ & dense & $ N $ & $L + L\sqrt{\log L / \log N}$ \\[0.3em]

        \citet{schmidthieber2019deeprelunetworkapproximation}; &
        Smooth manifold                                        &
        sparse                                                 &
        $N^2$                                                  & $\log N$                                                              \\

        \citet{chen2022nonparametric}                          &                       &       &                                       \\

        \citet{nakada2020adaptive}                             &
        Exact Minkowski-$d$                                    &
        sparse                                                 &
        $N^{2}$                                                & $O(1)$                                                                \\

        \citet{zhang2023effective}                             &
        Approx.~Minkowski-$d$                                  &
        sparse                                                 &
        $N^2$                                                  & $O(1)$                                                                \\

        \citet{kohler2023estimation}                           &
        Smooth manifold                                        &
        dense                                                  &
        $N$                                                    & $\log N$                                                              \\

        \citet{jiao2023deep}                                   &
        Smooth manifold                                        &
        dense                                                  &
        $N \log N$                                             & $L \log L$                                                            \\

                                                               &
        Approx.\ manifold                                      &
        dense                                                  &
        $N^{C} \log N$                                         & $L^{C} \log L$,                                                       \\

                                                               &
        Exact Minkowski-$d$                                    &
        dense                                                  &
        $N^{C} \log N$                                         & $L^{C} \log L$,                                                       \\

        \citet{Lu2021}                                         &
        $\mathcal M = [0,1]^d$                                 &
        dense                                                  &
        $N\log N$                                              & $L \log L$                                                            \\
        \bottomrule
    \end{tabular}
    \caption{Comparison of neural network approximation results for $s$-Hölder functions restricted to a set $\mathcal M\subset[0,1]^D$ with target accuracy $\lesssim (NL)^{-2s/d}$.
        Here $d$ denotes the intrinsic dimension of $\mathcal M$, $L = 1$ if unspecified,
        and $C \gg 1$ is a numerical constant.
    }
    \label{tab:approx-comparison}
\end{table}

The result improves and extends the literature in several ways; see \Cref{tab:approx-comparison} for a summary.
In particular, when $\Mcal$ has a general low-dimensional structure and the network is deep in the sense that $L \asymp \eps^{-\alpha}$, $\alpha > 0$, most previous results either do not apply or suffer from an exponentially worse dependence on $\eps$ ; see, e.g., \cite{jiao2023deep}. Even in comparatively simple cases, such as when $\Mcal$ is a smooth manifold or when $\Mcal = [0, 1]^D$, \Cref{thm:main-approx} improves upon known bounds by at least a logarithmic factor in the network width or depth.
Notably, even the strongest existing results for approximating $\Hcal_s([0, 1]^D, C)$ with $s > 1$ using deep, fully connected networks \citep{Lu2021} are sharpened by \Cref{thm:main-approx}, yielding logarithmic improvements in width and depth.

 In \Cref{thm:main-approx} we only require an upper bound on the covering numbers of $\Mcal$ at scale $\eps^{1/s}$, which is more general than assuming a bound on the Minkowski dimension. This is important when the domain $\Mcal$ is only approximately low-dimensional, for example, when the input lies close to a low-dimensional set but is perturbed by noise.
To illustrate, let $\Mcal \subset \R^D$ satisfy $\dim_M(\Mcal) = d$ and define its $\delta$-enlargement by $\Mcal^{+\delta} = \{\bx \in \R^D \colon \inf_{\by \in \Mcal}\| \bx - \by\|_\infty \le \delta\}$. Then $\mathcal{N}(\eps, \Mcal^{\delta+}, \|\cdot\|_{\infty}) \lesssim \eps^{-d}$ for all $\eps \ge 3\delta$, but $\dim_M(\Mcal) = D$. 
In Section \ref{sec:compositional}, this property will be essential for approximating compositions of functions defined on low-dimensional domains.

\subsection{Optimality}

In the special case of Lipschitz functions on $[0, 1]^D$, i.e., $f \in \Hcal_1([0, 1]^D, C)$, \citet{shen2022optimal,ou2024covering} show that a network of width $N / \sqrt{ \log N}$ and depth $L$ can approximate $f$ up to an error $(NL)^{-2/D}$. At first glance, this may suggest that the network sizes in our result could be improved by eliminating root-logarithmic factors in width or depth. However, such logarithmic factors are intrinsic to our setting and cannot, in general, be removed. 
Intuitively, approximating functions on a low-dimensional subset $\Mcal \subset [0, 1]^D$ with $\dim_M(\Mcal) = d < D$ is comparably difficult as approximating a function on $[0, 1]^d$. This explains why the required width and depth depend primarily on the intrinsic dimension $d$ rather than the ambient dimension $D$. However, using the same function class $\Fcal$ to approximate functions uniformly over all sets $\Mcal$ with $\dim_M(\Mcal) = d$ makes the problem slightly more challenging. This is formalized in the following proposition.

\begin{proposition} \label{prop:optimality}
    Let $\| f- g\|_{\Mcal}  = \sup_{\bx \in \Mcal} |f(\bx) - g(\bx)|$.
    \begin{enumerate}[(i)]
        \item For any function class $\Fcal$,
              \begin{align*}
                  \sup_{g \in \Hcal_s([0, 1]^D, C)}   \sup_{\dim_M(\Mcal) = d}   \inf_{f \in \Fcal} \| f- g\|_{\Mcal}  \lesssim \eps \quad \Rightarrow \quad \log \Ncal(\eps, \Fcal, \|\cdot\|_\infty) \gtrsim \eps^{-d/s} \log(\eps^{-1}).
              \end{align*}
        \item For any fixed $\Mcal$ with $\dim_M(\Mcal) = d$, there is a function class $\Fcal_\Mcal$ such that
              \begin{align*}
                  \sup_{g \in \Hcal_s([0, 1]^D, C)} \inf_{f \in \Fcal_\Mcal} \| f- g\|_{\Mcal} \lesssim \eps  \qquad \text{and} \qquad \log \Ncal(\eps, \Fcal, \|\cdot\|_\infty) \lesssim \eps^{-d/s} .
              \end{align*}
    \end{enumerate}
\end{proposition}

Combining \Cref{prop:optimality}~(i) with the covering number bounds for neural networks from \citet{ou2024covering}, we obtain the following optimality result for our approximation theorem.

\begin{proposition} \label{prop:optimality-bounded}
    The network size in \Cref{thm:main-approx} is optimal in the sense that
    \begin{align*}
        \log \Ncal(\eps, \wt \Fcal_{D, 1}(N, L + L \sqrt{\log L / \log N}, N), \|\cdot\|_\infty) \lesssim \eps^{-d/s} \log(\eps^{-1}).
    \end{align*}
\end{proposition}
Up to constants, it is not possible to construct smaller bounded-weight networks that achieve the same approximation accuracy uniformly over all sets $\Mcal$ with $\dim_M(\Mcal) = d$.

\subsection{Outline and main tools for the proof of Theorem \ref{thm:main-approx}}
The proof of \Cref{thm:main-approx}, given in Section \ref{apx:proofs}, follows the standard approximation strategy for Hölder-smooth functions used throughout the literature. Suppose $\Mcal \subseteq [0, 1]^D$, fix $K \sim \eps^{-1/s}$, and partition $[0,1]^D$ in a regular grid with mesh size $1/K$ and grid points $\mathcal{G}_K=\{\bz_{\bbeta}=\bbeta/K: \bbeta \in \{0, \dots, K-1\}^D\}$. For $\bx \in [0, 1]^D$, define $\bbeta(\bx):=(\lfloor K x_j \rfloor)_{j=1}^D$ and let $\bx_{\bbeta} = \bbeta(\bx)/ K$ denote the grid point of the cell containing $\bx$. By Taylor's theorem, 
          \begin{align*}
              \sup_{\bx \in [0, 1]^D} \left|f(\bx) - \sum_{\substack{\|\balpha\|_1 \leq \lfloor s \rfloor }} \frac{\partial^{\balpha} f(\bx_{\bbeta})}{\balpha!} (\bx - \bx_{\bbeta})^{\balpha} \right| \lesssim \eps.
          \end{align*}
Thus, uniformly approximating $f$ over $[0,1]^D$ reduces to approximating piecewise polynomials of degree $\lfloor s \rfloor$, with coefficients given by derivatives of $f$ evaluated at the grid points. The network construction proceeds in stages: first approximating the map $\bx \mapsto \bx_{\bbeta}$, then implementing the truncated Taylor polynomial. Proposition 4.1 of \citet{Lu2021} (see also \Cref{lem:poly-approx}) provides an efficient approximation of the monomial basis, reducing the remaining task to fitting coefficients $\partial^{\balpha} f(\bx_{\bbeta}) / \balpha!$ at the grid points $\bx_{\bbeta}$. Constructing networks that memorize these values and combining all parts yields the desired approximation.

\medskip
\textit{Differences to previous proof strategies.} Compared to prior work, our proof systematically uses recent tools for approximating functions with dense, bounded-weight networks with flexible width and depth  \citep{Lu2021, ou2024three, ou2024covering}.
In particular, approximating the map $\bx \mapsto \bx_{\bbeta}$ and the associated monomials is straightforward (\Cref{lem:stepfun-2}, \Cref{lem:poly-approx}) and can be achieved with networks significantly smaller than the final network in \Cref{thm:main-approx}. Moreover, we adapt the median smoothing technique of \citet{Lu2021} to sets of low intrinsic dimension (\Cref{prop:median-smoothing-nn}), allowing an approximation that was initially valid only on the set where $\bx \mapsto \bx_{\bbeta}$ is computed exactly to be extended to the entire set $\Mcal$.

The main challenge lies in efficiently solving the point fitting problem. 
When $\Mcal = [0, 1]^D$, this is typically handled by re-indexing the grid points $\bz_{\bbeta}$ to $\{1, \ldots, K^D\}$ and using bit-extraction techniques to memorize the corresponding coefficients \citep[e.g.,][]{shen2022optimal,Lu2021,ou2024three}.
For a general set $\Mcal$ with Minkowski dimension $d < D$, this approach becomes inefficient. Only grid cells $Q_{\bbeta} = \bz_{\bbeta} + [0, 1/K]^D$ intersecting $\Mcal$ need to be stored, and there are only $O(K^d)$ such $\bbeta$. However, there is no simple reindexing from the full grid $\{0, \ldots, K-1\}^D$ that uniquely identifies these relevant cells.
\citet{nakada2020adaptive} circumvented this problem by embedding the Taylor coefficients directly into the network weights using extremely wide architectures. While effective in theory, this results in sparse, very wide and shallow networks that are far from practical implementations.

Our key technical innovation is a new result on network memorization that solves the problem efficiently with dense networks of flexible width and depth.
\begin{proposition} \label{prop:point-fitting-ours}
    Let $N, L, s \in \N$, $\delta > 0$, and suppose we are given samples
    $$(\bx_1, y_1), \dots, (\bx_J, y_J) \in [0, 1]^D \times \{0, \dots, 2^r - 1\},$$
    where $\|\bx_i - \bx_j\| \ge \delta$ for every $i \neq j$ and $J \le N^2 L^2$. Denote $R = 2J^2 D /\delta $. Then there is 
    $$\phi \in \wt \Fcal_{D, 1}\left(N,  L + L \frac{\sqrt{\log L}}{\sqrt{\log N}}  + L\frac{\log(\delta^{-1}) +  r}{\sqrt{\log N \log (NL)}}, N + 2^{(r + \log (R)) / L} \right) $$
    such that
        $\phi(\bx_j) =  y_j \quad \text{for } j = 1, \dots, J.$
\end{proposition}
The proof is given in \Cref{apx:proofs} and outlined as follows. First, implement an affine map $z\colon [0, 1]^D \to [0, 1]$ that keeps the values $z(\bx_j)$ sufficiently separated. Then implement a piecewise linear map that assigns the $z(\bx_j)$ to crafted numbers $u, w$. These numbers encode the necessary binary digits of nearby $z(\bx_k)$ values along with the corresponding $y_k$ values in a reduced-alphabet ternary expansion. Using techniques developed in \citet{ou2024three}, the values $z(\bx_k)$ and $y_k$ can then be decoded from $u, w$ using flexible bounded-weight networks (\Cref{lem:bit-decode}). It remains to compare $z(\bx_j)$ to the stored $z(\bx_k)$ values and return the corresponding $y_j$ when a match is found.

Our proof strategy shares some similarity with the best previous result on network memorization \citep[Theorem 5.1]{vardi2022on}, but makes more efficient use of network width and weight magnitude. In particular, \Cref{prop:point-fitting-ours} improves their result from width $N^2$ to $N$, reduces the depth by a $\sqrt{\log N}$ factor, and the weight magnitude from about $2^{L (r + \log(R))}$ to $N + 2^{(r + \log(R))/L}$.

\section{Compositional models with low Minkowski dimension}
\label{sec:compositional}

In this article we consider a model which combines the compositional model introduced in \cite{kohler2021rate} with the low Minkowski dimension assumption considered in \cite{nakada2020adaptive}. More precisely, we consider a class of compositional functions in which each component function either depends on only a small number of input coordinates or is defined on a set of intrinsically low dimension, quantified via its Minkowski dimension.

\begin{definition}[Class of compositional models on sets with low Minkowski dimension]
    \label{def1}
    A function \( f \colon \mathbb{R}^D \supset \Mcal \to \mathbb{R} \) is called a \emph{compositional model of level}
    \(\ell \in \mathbb{N}\) with dimension vector $(D_0, \dots, D_\ell)$, where $D_0=D$ and $D_\ell=1$ and \emph{smoothness and intrinsic dimension parameter set}
    \[
        \mathcal{P} = \{(d_{ij}, s_{ij}) \in \mathbb{N} \times \mathbb{R}_{\geq 1} \colon i = 0, \dots, \ell,\ j = 1, \dots, D_{i}\},
    \]
    if there exist functions
        $g_i \colon [-C, C]^{D_i} \to [-C, C]^{D_{i+1}}$, $i = 0, \dots, \ell$,   for some $C>0$,
    such that $g_0 = \text{id}$ and
    \begin{align}
        \label{comp}
        f = g_{\ell} \circ g_{\ell-1} \circ \dots \circ g_0,
    \end{align}
    where \( g_i \) is a vector-valued function with components
        $g_i = (g_{i1}, \dots, g_{iD_{i+1}})^{\top}$
    satisfying the following conditions.
    \begin{itemize}
        \item[\textnormal{(C)}] (\emph{Coordinate sparsity}) There exists a subset $S_{ij} \subseteq \{1, \dots, D_i\}$ such that $g_{ij}(\bx)=g_{ij} \circ \pi_{S_{ij}}(\bx)$, where $\pi_{S_{ij}}(\bx) =(x_k)_{k \in |S_{ij}|}$ denotes the coordinate projection.
        \item[\textnormal{(S)}] (\emph{Smoothness}) It holds \( g_{ij} \in \mathcal{H}_{s_{ij}}([-C, C]^{|S_{ij}|}, C) \).
        \item[\textnormal{(M)}] (\emph{Low Minkowski dimension}) The set $\mathcal{M}_{ij} := \pi_{S_{ij}} \circ g_{i - 1}(\Mcal)$  satisfies $\dim_M({\mathcal{M}_{ij}}) \leq d_{ij}$.
    \end{itemize}
    The corresponding class of compositional models is defined as
    \begin{align*}
        \mathcal{G}(\ell, \mathcal{P},  \bm{D})=\{f \text{ of the form }\eqref{comp}\}.
    \end{align*}
\end{definition}

\begin{remark} 
    The function classes considered in \cite{kohler2021rate, schmidt2020nonparametric}
    arise as special cases of Definition~\ref{def1} in which only coordinate sparsity \textnormal{(C)} and smoothness \textnormal{(S)} is imposed on the component functions.
    Conversely, \cite{nakada2020adaptive} studies smooth function classes characterized by
    low Minkowski dimension, which are a special case of Definition~\ref{def1} with $\ell = 1$.
\end{remark}

The modeling assumption in Definition~\ref{def1} is motivated by the structure of many real-world data-generating processes. In many applications, the target function can be decomposed into a sequence of simpler transformations, where at each stage of the hierarchy the component functions may depend only on a small subset of coordinates, reflecting sparse or localized interactions. At the same time inputs are often constrained to intrinsically low-dimensional sets due to latent factors, dependencies among covariates, or physical and biological constraints. By allowing either coordinate sparsity or low Minkowski dimension at each component or both, Definition~\ref{def1} captures these common structural features of regression functions while remaining flexible enough to encompass classical additive and index models as well as more complex hierarchical dependencies.


\begin{example}[Sparse feature extraction with latent factor aggregation]
    Let $\mathcal{X} \subset \R^D$ be a compact set. Consider a function of the form
    \begin{align*}
        f(\bx) = h(g_1(\bx_{S_1}), \dots, g_m(\bx_{S_m})),
    \end{align*}
    where each $S_j \subset \{1, \dots, D\}$ with $|S_{j}| \leq d$ and $\bx_{S_j} = (x_{k})_{k \in S_j}$. If all component functions $g_j:\R^{|S_{\ell}|} \to \R$ are Lipschitz, the feature map
    \begin{align*}
        \bx \mapsto G(\bx) = (g_1(\bx_{S_1}), \dots, g_m(\bx_{S_m})) \in \R^m
    \end{align*}
    is Lipschitz and, by standard properties of Minkowski dimension under Lipschitz maps, we have
    \begin{align*}
        \dim_M(G(\mathcal{X})) \leq \min\left\{ \dim_M(\mathcal{X}), m\right\}.
    \end{align*} Hence, $G(\bx)$ takes values in a set of intrinsically low dimension. As the second-level function $h$ acts on this set, it therefore satisfies the low Minkowski dimension condition \textnormal{(M)}.
\end{example}

Approximating a compositional model reduces to approximating its individual component functions. The global approximation error can be controlled by the largest component-wise approximation error, as shown in the following lemma.

\begin{lemma}[Error propagation lemma]
\label{lemprop} \label{lem:error-propagation:1}
    Let $h=h_{\ell} \circ \cdots \circ h_0$ be a compositional model defined as in \Cref{def1}. Let $\hat{h}$ be obtained by replacing each $h_{ij}$ by $\hat{h}_{ij}$. Define partial compositions
    \begin{align*}
        H_i:=h_i \circ \cdots \circ h_0, \quad \hat{H}_i:=\hat{h}_i \circ \cdots \circ \hat{h}_0,
    \end{align*}
    and the accumulated errors and
    \begin{align*}
        \delta_0 = 0, \qquad \delta_i:=\sup_{x \in \mathcal{M}} \|H_i(\bx)-\hat{H}_i(\bx)\|_{\infty}, \quad i \ge 1.
    \end{align*}
    For each $i=0, \dots, \ell, j=1, \dots, D_i$, let $\mathcal{M}^{+\delta_{i-1}}_{ij}$ be the $\delta_{i-1}$-enlargement of $\mathcal{M}_{ij}$ defined as $\mathcal{M}_{ij}^{+\delta_{i-1}} := \{\bx \in \R^{d_{ij}}: \inf_{\by \in \mathcal{M}_{ij}}(\bx,\by) \leq \delta_{i-1}\}$
and assume that 
    \begin{align}
    \label{error}
        \sup_{\bz \in (\mathcal{M}_{ij})^{+\delta_{i-1}}} |h_{ij}(\bz)-\hat{h}_{ij}(\bz)| \leq \epsilon_{ij}.
    \end{align}
    Then  $ \delta_i \leq C \delta_{i-1} + \max_j \epsilon_{ij} $ and 
    \begin{align*}
        \|h-\hat{h}\|_{\infty} = \delta_{\ell} \leq \sum_{i=1}^{\ell} C^{\ell-i} \max_j \epsilon_{ij} \leq \ell \max\{C^{\ell-1}, 1\} \max_{ij} \epsilon_{ij}.
    \end{align*}
   
\end{lemma}
\begin{proof}
    Fix $\bx \in \mathcal{M}$. For each $i \geq 1$ and component $j$, let
    \begin{align*}
        \bu := \pi_{S_{ij}}(H_{i-1}(\bx)) \in \mathcal{M}_{ij}, \quad \hat{\bu} := \pi_{S_{ij}}(\hat{H}_{i-1}(\bx)).
    \end{align*}
    Then $\|\bu-\hat{\bu}\|_{\infty} \leq \|H_{i-1}(\bu)-\hat{H}_{i-1}(\bu)\|_{\infty} \leq \delta_{i-1}$, so $\hat{\bu} \in (\mathcal{M}_{ij})^{+\delta_{i-1}}$. Using triangle inequality together with the Lipschitzness of $h_{ij}$ and \eqref{error}, yields
    \begin{align*}
        |h_{ij}(\bu)-\hat{h}_{ij}(\hat{\bu})| &\leq |h_{ij}(\bu)-h_{ij}(\hat{\bu})|+|h_{ij}(\hat{\bu})-\hat{h}_{ij}(\hat{\bu})|
        \leq C \|\bu-\hat{\bu}\|_{\infty} + \epsilon_{ij} \leq C\delta_{i-1} + \epsilon_{ij},
    \end{align*}
     yielding $\delta_i \leq C\delta_{i-1} + \max_j \epsilon_{ij}$. Unrolling the recursion proves the bound.
\end{proof}


We note that many functions can be expressed as a composition of simpler functions and thus form a compositional model of corresponding level $\ell$ according to Definition \ref{def1}. For instance,
\begin{align*}
    g(x,y):= xy = \frac{(x+y)^2 - (x-y)^2}{4} = g_1 \circ g_0,
\end{align*}
where $g_{01}(x,y) = (x+y)^2$ and $g_{02}(x,y)=(x-y)^2$ and $g_1(z_1, z_2)=(z_1 - z_2)/4$.
Here $D=D_1=2$ and $D_0=d_{01}=d_{02}=2$. In turn, $xy$ can be considered as a compositional model of level $1$. The previous lemma shows, that if we have given approximators for every function in the composition, then the composition of these approximators forms also an approximator for the composition model.
In general, most constructive approximation results for neural networks employ an underlying compositional structure, even when the target function is not explicitly represented in this form. Such a structure is also used to propagate approximation errors in the proof of \Cref{thm:main-approx}. 

For approximating compositional functions as defined in \Cref{def1}, we now combine \Cref{thm:main-approx} with \Cref{lem:error-propagation:1}. The following theorem shows  that the networks required to approximate the compositional model are only slightly larger than those for a single, $s^*$-smooth function $g_{ij}$ with $d^*$-dimensional domain. To clearly capture the enlargement, we leave constants depending on $\ell$ explicit.


\begin{theorem} \label{thm:main-compositional}
    Let $\eps > 0$, $S = [a_j, b_j]_{j = 1}^D$, $\Mcal \subseteq S$, $g \in \Gcal(\ell, \Pcal, \bm D)$, and 
    $$(d^*, s^*) \in \arg\max_{(d, s) \in \Pcal} d/s.$$
    Then there are $N, L \in \N$ with $N^2L^2 \asymp \eps_0^{-d^*/s^*} := \left(\eps/\ell(C + 1)^\ell \right)^{-d^*/s^*}$ and a neural network
    \begin{align*}
        \phi \in \wt \Fcal_{D, 1}(\max_{i} D_iN, \ell(L + L \sqrt{\log L / \log N}), N) \quad \text{such that} \quad  \sup_{\bx \in \Mcal} |g(\bx) - \phi(\bx)| \lesssim \eps.
    \end{align*}
\end{theorem}

\textbf{Proof sketch}
Fix a target accuracy $\eps>0$. Fix $\eps_0>0$ to be chosen later and define a geometrically increasing error target
$\eps_i := \eps_0 (C+1)^i , i=0,\dots,\ell.$
Assume $\eps_\ell\le 1$ and set
$L' := L + L\sqrt{\log L/\log N}.$
As in \Cref{lem:error-propagation:1}, define
\[
G_i := g_i\circ\cdots\circ g_0, \quad
\hat G_i := \hat g_i\circ\cdots\circ \hat g_0, \quad \delta_0 := 0, \quad 
\delta_i := \sup_{\bx\in\Mcal}\|G_i(\bx)-\hat G_i(\bx)\|_\infty .
\]

\medskip
\noindent\emph{Step 1: Approximation on fixed enlargements.}
For each level $i=0,\dots,\ell$ and each component $g_{ij}$, approximate $g_{ij}$ uniformly on the enlarged set $(\Mcal_{ij})^{+\eps_i}$ with accuracy $\eps_i$.  
Using \Cref{thm:main-approx}, this is possible with networks
\[
\hat g_{ij}\in \tilde{\Fcal}_{D_{i-1},1}(N,L',N), \quad \text{s.t. } N^2L^2 \asymp \eps_i^{-d_{ij}/s_{ij}}.
\]
Since $\eps_i$ increases in $i$ and $\eps_0 \le 1$, the single condition $N^2L^2 \asymp \eps_0^{-d^*/s^*}$ works for all $i,j$.

\medskip
\noindent\emph{Step 2: Error propagation.}
We show inductively that for all $i=0,\dots,\ell$,
\[
\delta_i \le \eps_{i+1},
\qquad 
\hat G_i \in \tilde{\Fcal}_{D,D_i}(\max_k D_k N,\, iL',\, N).
\]
The base case $i=1$ follows since $\delta_0=0$ and the approximation of $g_{1j}$ applies on $(\Mcal_{1j})^{+\eps_1}$.  
For the inductive step, assume $\delta_{i-1}\le\eps_i$. Then the inputs to $g_{ij}$ produced by $\hat G_{i-1}$ lie in $(\Mcal_{ij})^{+\eps_i}$, where the local approximation error is at most $\eps_i$. By the error propagation lemma (\Cref{lem:error-propagation:1}),
\[
\delta_i \le C\delta_{i-1}+\eps_i
\le (C+1)\eps_i
= \eps_{i+1}.
\]
The stated network size follows from standard parallelization and composition arguments.

\medskip
\noindent\emph{Step 3: Choice of $\eps_0$.}
\Cref{lem:error-propagation:1} yields
\[
\delta_\ell 
\le \sum_{i=1}^\ell C^{\ell-i}\eps_i
\le \ell (C+1)^\ell \eps_0 .
\]
Choosing
$\eps_0 := \eps/\ell (C+1)^\ell$ 
ensures 
\[
\delta_\ell  = \sup_{\bx \in \Mcal} | G_\ell(\bx) - \hat G_\ell(\bx)| ) = \sup_{\bx \in \Mcal} | g(\bx) - \phi(\bx)| \le \eps,
\]
as claimed. Full details can be found in \Cref{apx:proofs}.
\hfill $\blacksquare$ 


\section{Application to nonparametric regression}
\label{sec:stat}

We finally apply our approximation results to derive statistical rates in a nonparametric regression problem.
Suppose we observe $n$ independent and identically distributed pairs of covariates and corresponding responses $(\bX_i, Y_i)$ from the model
\begin{align*}
    Y_i = f_0(\bX_i) + \epsilon_i, \quad i=1, \dots, n
\end{align*}
with $\epsilon_i \sim \mathcal{N}(0, \sigma^2)$ independent of $\bX_i$. We want to learn the regression function $f_0: \R^D \to \R$ by empirical risk minimization over a class of neural networks. To keep the proofs simple, we only consider the square loss and Gaussian noise, noting that the same rates can be derived in more general loss and noise settings with \citet[Theorem 3.4.1]{van2023weak}.

\begin{theorem} \label{thm:main-stat}
    Let $s > 0$, $d \in \N$, and define $\eps_n = n^{-s/(2s + d)} \log(n)^{s/(2s + d)}$. 
    Choose $N_n, L_n \in \N$ such that $N_n^2L_n^2 \asymp  \eps_n^{-d/s}$ and assume the class $ \Fcal_n = \wt \Fcal_{D, 1}(N_n, L_n + L_n \sqrt{\log L_n / \log N_n}, N_n)$ satisfies $\inf_{f \in  \Fcal_n} \|f - f_0\|_{\infty} \lesssim \eps_n$. Then  any empirical risk minimizer 
    \begin{align*}
         \wh f \in \argmin_{f \in \mathcal{F}_n} \frac{1}{n} \sum_{i=1}^n (Y_i - f(\bX_i))^2,
    \end{align*}
    satisfies
    \begin{align*}
         \E\left[\| \wh f - f_0\|_{L^2(\P_{\bX})}^2\right] \lesssim \eps^2_n.
    \end{align*}
\end{theorem}

\begin{proof}
Theorem 4.1 of \citet{ou2024covering} gives the oracle inequality
    \begin{align*}
         \E\left[\| \wh f - f_0\|_{L^2(\P_{\bX})}^2\right] \lesssim \eps_n^2 + \frac{1 + \log \Ncal(\eps_n^2, \Fcal_n, \| \cdot \|_\infty)}{n}.
    \end{align*}
    By \Cref{prop:optimality-bounded}, we have 
\begin{align*}
     \log \Ncal(\eps_n^2, \wt \Fcal_n, \|\cdot\|_{\infty}) \lesssim \eps_n^{-d/s} \log(\eps_n^{-1}),
\end{align*}
and it holds
\begin{align*}
    \eps_n^{-d/s} \log(\eps_n^{-1})  =  \eps_n^2 \eps_n^{-(2s + d)/s} \log(\eps_n^{-1})  = n \eps_n^2  \log(\eps_n^{-1}) / \log(n) \lesssim n \eps_n^2.
\end{align*}
Combining the displays and noting $\eps_n^2 \gtrsim 1/n$, the claim follows.
\end{proof}

In case that $f_0$ is $s-$smooth, applying \Cref{thm:main-approx} with $\Mcal = [0, 1]^D$ gives a statistical rate $\eps_n = n^{-2s/(2s + D)} \log(n)^{2s/(2s + D)}$, which matches the minimax optimal rate up to a small $ \log(n)^{2s/(2s + D)}$ factor \citep{stone1982optimal}. Notably, even in this classical setting, our result improves the best known results for fully connected networks by at least a $(\log n)^4$ factor (\cite{jiao2023deep, kohler2021rate, kohler2023estimation}). This improvement is more relevant than it may seem: for example, taking $s=1$ and $D=10$, the sequence $\eps_n' = n^{-2s/(2s + D)} (\log n)^{4}$ increases until $n \ge 10^{10}$, while the rate $\eps_n$ implied by \Cref{thm:main-stat} decreases for all $n \ge 3$.
When the domain $\Mcal$ has Minkowski dimension $d \ll D$, however, we obtain a much faster statistical rate of order $n^{-2s/(2s + d)} \log(n)^{2s/(2s + d)}$, also matching the minimax rate up to log factors.
Several previous works achieve a similar rate (with worse $\log$-factors), but do not apply to deep networks. 

Lastly, combining \Cref{thm:main-stat} with \Cref{thm:main-compositional} 
gives even faster statistical rates for compositional models. The rate $\eps_n = n^{-2s^*/(2s^* + d^*)} \log(n)^{2s^*/(2s^* + d^*)}$ depends only on the pair $(s^*, d^*) \in \Pcal$ attaining the maximal ratio $d^*/s^*$. Especially if the input map $g_1$ is very smooth (e.g., affine) and the composition is sparse, this rate can be way faster than a general model for $g \in \Hcal_{\min_{ij}s_{ij}}(\Mcal, C)$ ignoring this structure. This shows that deep neural networks can adapt efficiently to compositional structures and low intrinsic dimensionality, providing further explanation for their practical success.

\section{Discussion}

A recurring theme in neural network approximation theory is the close alignment between architectural design and structural assumptions on the target function. Feedforward neural networks are inherently compositional objects, and \Cref{lem:error-propagation:1} makes explicit how approximation guarantees for compositional target functions can be obtained by combining approximation bounds for their individual components. From this perspective, depth enables a controlled propagation of approximation errors across hierarchical representations, allowing the overall approximation complexity to be governed by the most challenging components in the composition.

The compositional model class introduced in \Cref{def1} is sufficiently flexible to capture structural properties commonly associated with convolutional architectures. In image classification tasks, target functions are often assumed to exhibit locality and hierarchical aggregation, where decisions are based on features extracted from spatially localized regions and progressively combined across layers. Within our framework, such behavior can be modeled by component functions depending only on subsets of input coordinates, corresponding to spatial locality, while low intrinsic dimensionality reflects the empirical observation that intermediate feature representations concentrate on low-dimensional sets. Similar structural assumptions have previously been used to derive sharp statistical guarantees for convolutional classifiers \citep{KOHLER2025106188}. While our analysis focuses on fully connected networks, it suggests that comparable approximation guarantees may be attainable for convolutional architectures under suitable notions of local compositionality.

A related perspective applies to transformer-based models for language tasks. In this setting, inputs are sequences of tokens rather than spatial grids, and interactions between tokens can span long ranges. Nonetheless, many language understanding tasks appear to admit a compositional structure in which representations are constructed from local token interactions and subsequently refined through increasingly global contextual integration. Empirical studies indicate that intermediate representations in transformer models often concentrate on low-dimensional sets and that only a limited subset of token interactions effectively contributes to each representation \citep{ValerianiDCLAC23}. Extending the present framework to this setting would require replacing spatial locality with appropriate structural constraints on token interactions, providing an interesting direction for future theoretical work.

Overall, many modern learning architectures rely on feedforward networks as core building blocks, and many learning problems exhibit structural properties such as compositionality and low intrinsic dimensionality. The approximation theory developed in this work for flexible-width-and-depth feedforward networks provides a powerful foundation for analyzing such settings. At the same time, the proposed compositional model class offers a way to formalize prediction tasks that combine hierarchical structure with geometric constraints, suggesting several directions for extending approximation and learning guarantees to more specialized architectures.

\bibliography{hcm.bib}

\appendix

\section{Proofs of the main results} \label{apx:proofs} 

Throughout our proofs we frequently make use of the following properties of ReLU networks:

\medskip\noindent\textit{Identity.} We repeatedly use the fact that the identity function on $\R$ can be represented using ReLU activation as
 $       x=\max\{x,0\}-\max\{-x,0\}=\sigma(x)-\sigma(-x) \in \mathcal{F}_{1,1}((2,1),2,1)$.
    Accordingly, for $\bx \in \R^D$,
    \begin{align*}
        \bx=(\sigma(x_1)-\sigma(-x_1), \dots, \sigma(x_1)-\sigma(-x_1))^T =:f_{id}(\bx) \in \mathcal{F}_{D,D}((2D,D),2,1).
    \end{align*}
    
\medskip\noindent\textit{Depth alignment.} We can align the depth of two networks $f \in \mathcal{F}_{d_0, d_f}(N_f, L_f, B)$ and $g \in \mathcal{F}_{d_0, d_g}(N_g, L_g, B)$ with $L_f < L_g$ by using that 
$$f=\underbrace{f_{id} \circ \dots \circ f_{id}}_{L_g - L_f} \circ f \in \mathcal{F}_{d_0, d_f}(N_f \vee 2d_f,  L_g, B).$$

\medskip\noindent\textit{Parallelization.} After depth alignment, we have 
        $(f,g) \in \mathcal{F}_{d_0, d_f+d_g}((N_{f} \vee 2d_L)+N_g, L_g, B)$.
        
\medskip\noindent\textit{Composition.} Composing $f$ and $g$ leads to $f \circ g \in \mathcal{F}_{d_0, d_f}(\max\{N_g, N_f\}, L_f + L_g, B)$. Here we do not need to apply depth alignment before.

\subsection{Proof of \Cref{thm:main-approx}}
The affine rescaling $A(\bx) = \left((x_j - a_j)/(b_j - a_j)\right)_{j=1}^D$ can be implemented exactly by a zero-hidden-layer network. Since this transformation introduces no approximation error, and since approximating $f / C$ with error $\eps$ is equivalent to approximating $f$ with error $C \eps$, we may assume without loss of generality that $f \in \Hcal_s([0, 1]^D, 1)$. The proof proceeds in four steps. 

\begin{enumerate}
    \item We partition $[0,1]^D$ in a regular grid of cubes 
    \begin{align*}
              Q_{\bbeta} = \left[ \frac{\beta_i}{K}, \frac{\beta_{i} +1}{K}\right]_{i = 1}^D, \quad \bbeta \in \{ 0, \dots, K - 1 \}^D
          \end{align*}
          of side length $1/K$. We choose $K \sim (N^2 L^2 )^{1/d}  \sim \eps^{-1/s}$ so that $\mathcal{M} $ is covered by at most $K^{d} \sim  \eps^{-d/s}$ such cubes. For $\bx \in Q_{\bbeta}$, we denote by $\bx_{\bbeta} = \bbeta / K$ the lower left corner of the cube containing $\bx$ and let $\bx_{\bbeta}'$ be an element on the line segment between $\bx$ and $\bx_{\bbeta}$. By Taylor's theorem, we have
          \begin{align*}
              f(\bx) = \sum_{\substack{\|\balpha\|_1 \leq \lfloor s \rfloor}} \frac{\partial^{\balpha} f(\bx_{\bbeta})}{\balpha!} (\bx - \bx_{\bbeta})^{\balpha} + \sum_{\substack{\|\balpha\|_1 = \lfloor s \rfloor }} \frac{\partial^{\balpha} f(\bx_{\bbeta}') - \partial^{\balpha} f(\bx_{\bbeta})}{\balpha!} (\bx - \bx_{\bbeta})^{\balpha}.
          \end{align*}
          In turn, approximating $f(\bx)$ reduces to approximating the truncated Taylor polynomial on each cube, with a remainder term controlled by the Hölder smoothness of $f$.
    \item We now approximate the truncated Taylor sum by a neural network.
          \begin{itemize}
              \item[(i)] By \Cref{lem:poly-approx}, for every $\balpha=(\alpha_1, \dots, \alpha_D)^T$ there is a neural network
                  \begin{align*}
                      \phi_{P, \balpha} \in \wt \Fcal_{D, 1}(N, L, 1),
                  \end{align*}
                  such that 
                  \begin{align*}
                      \sup_{\bx \in [0, 1]^D}| \phi_{P, \balpha}(\bx) - x_1^{\alpha_1} \cdots x_D^{\alpha_D} |\lesssim N^{-7sL} \lesssim (N^2 L^2)^{-s/d} \lesssim \eps.
                  \end{align*}

              \item[(ii)] Set $\delta = \eps / 4$. By \Cref{lem:stepfun-2}, there exists a network
                  \begin{align*}
                      \phi_{\beta} \in \wt \Fcal_{1, 1}(N, L,N),
                  \end{align*}
                  such that
                  \begin{align*}
                      \phi_{\beta}(x) = \beta / K  \quad \text{for } x \in \left[\frac{\beta}{K}, \frac{ \beta + 1}{K}- \delta \ind_{\beta \le K - 2}\right], \quad \beta = 0, \dots K - 1,
                  \end{align*}
                  and $\phi_\beta(x) = 0$ for $x \le 0$ and $\phi_\beta(x) = (K - 1) / K$ for $x \ge 1$.
                  By parallelizing these networks, i.e., computing $(\phi_{\beta_1}, \dots, \phi_{\beta_D})$, we construct a network
                  \begin{align*}
                      \phi_{\bbeta} \in \wt \Fcal_{D, 1}(N, L,  N),  
                  \end{align*}
                  such that
                  \begin{align*}
                      \phi_{\bbeta}(\bx) = \bx_{\bbeta} \quad \text{for } \bx \in \left[\frac{\beta_i}{K}, \frac{ \beta_i + 1}{K}- \delta \ind_{\beta_i \le K - 2}\right]_{i = 1}^D = [0,1]^D \setminus \Omega_{K, \delta},
                  \end{align*}
                  where
                  \begin{align}
                      \Omega_{K, \delta} = \bigcup_{j = 1}^D \left\{\bx \in [0,1]^d\colon x_j \in \bigcup_{k = 1}^{K - 1} \left(\frac{k}{K} - \delta,  \frac{k}{K}\right)\right\}. \label{eq:omega}
                  \end{align}

              \item[(iii)]
                  Define $\xi_{\bbeta, \balpha} = \partial^{\balpha} f(\bbeta/K) / \balpha!$ for all $\bbeta \in \{0, \dots, K - 1\}^D$ such that $Q_{\bbeta} \cap \Mcal \neq \emptyset$.
                  We want to approximate these values by a neural network up to accuracy $\eps$ using \Cref{prop:point-fitting-ours}.
                  Set
                  \begin{align*}
                      \wt r &= \lceil \log(\eps^{-1}) \rceil \asymp \log(NL), \\
                      \wt J      & \asymp \eps^{-d/s} \asymp  N^2 L^2,         \\
                      \wt \delta & \asymp \eps^{1/s} \asymp ( N^2 L^2)^{-1/d}, \\
                      \log(\wt R) & = \log(2D \wt J^2/\wt \delta)\asymp \log(NL).
                  \end{align*}
                  Observe that $\log(\wt \delta^{-1}) + \wt r \lesssim \log(NL)$ and $\log (NL) / \sqrt{\log N \log (NL)} \lesssim 1 + \sqrt{\log L / \log N}$.
                  Then according to \Cref{prop:point-fitting-ours}, there exists a neural network
                  \begin{align*}
                      \phi_{\xi, \balpha} \in \wt \Fcal_{D, 1}\left(N,  L+ L\sqrt{\log L / \log N}, N + 2^{(\wt r + \log(\wt R)) / L} \right),
                  \end{align*}
                  such that
                  \begin{align*}
                      \phi_{\xi, \balpha}(\bx_{\bbeta}) = 2^{-\wt r}\lfloor 2^{\wt r} \xi_{\bbeta, \balpha} \rfloor  \quad \text{for every } \bbeta \text{ s.t. } Q_{\bbeta} \cap \Mcal \neq \emptyset.
                  \end{align*}
                  Which implies
                  \begin{align*}
                      |\phi_{\xi, \balpha}(\bx_{\bbeta}) - \xi_{\bbeta, \balpha}  | \lesssim 2^{- \wt r} = \eps \quad \text{for every } \bbeta \text{ s.t. } Q_{\bbeta} \cap \Mcal \neq \emptyset.
                  \end{align*}
                  Since 
          \begin{align*}
            \log(\wt R) / L \lesssim \log(NL) / L \lesssim \log N,
          \end{align*}
          we have 
          \begin{align*}
            \phi_{\xi, \balpha} \in \wt \Fcal_{D, 1}\left(N,  L+ L\sqrt{\log L / \log N}, N  \right).
          \end{align*}
          \end{itemize}
          
    \item Combining the above, we construct a neural network
          \begin{align*}
              \tilde \phi \in \wt \Fcal_{D, 1}\left(N,  L+ L\sqrt{\log L / \log N} , N \right),
          \end{align*}
          defined as 
          \begin{align*}
              \tilde \phi(\bx) = \sum_{\substack{\|\balpha\|_1 \leq \lfloor s \rfloor}} \phi_{P, (1, 1)}( \phi_{\xi, \balpha}(\bx_{\bbeta}) , \phi_{P, \balpha}(\bx - \phi_{\bbeta}(\bx))).
          \end{align*}
          By \Cref{lem:error-propagation:1} and the error bounds on the individual networks, we have
          \begin{align*}
              \sup_{\bx \in \Mcal_K \setminus \Omega_{K, \delta}} | f(\bx) - \tilde \phi (\bx)| \lesssim \eps,
          \end{align*}
          where $\Mcal_K = \bigcup_{\bbeta\colon Q_{\bbeta} \cap \Mcal \neq \emptyset} Q_{\bbeta} $.
    \item Finally, \Cref{prop:median-smoothing-nn} implies that there is $\phi \in \wt \Fcal_{D, 1}\left(N, L+ L\sqrt{\log L / \log N}, N \right)$ such that
          \begin{align*}
              \sup_{\bx \in \Mcal_K} | f(\bx) -  \phi (\bx)| \lesssim \eps.
          \end{align*}
          The final claim follows since $\Mcal \subset \Mcal_K$. 
\end{enumerate}

\subsection{Proof of \Cref{prop:optimality}}

We start with an auxiliary lemma used for the first part of the proposition.
The following is a modified version of Proposition 3.1 of \citet{ou2024covering}.
\begin{lemma} \label{lem:covering-lower-bound}
    Let $\Gcal$ be a class of functions and $\Scal$ be a collection of sets $S \subseteq [0, 1]^D$ such that
    \begin{align*}
        \inf_{f, g \in \Gcal, f \neq g} \sup_{S \in \Scal} \|f - g\|_{S} \ge 3\eps.
    \end{align*}
    Then any function class $\Fcal$ with
    \begin{align*}
        \sup_{g \in \Gcal}\sup_{S \in \Scal} \inf_{f \in \Fcal} \|f - g\|_{S} \le \eps,
    \end{align*}
    must satisfy $\Ncal(\eps, \Fcal, \|\cdot\|_\infty) \ge |\Gcal|$.
\end{lemma}
\begin{proof}
    We prove this by contradiction. Suppose there exist $f_1, \dots, f_N \in \Fcal$ with $N < |\Gcal|$ such that for any $f \in \Fcal$, there exists an $i$ with $\|f - f_i\|_\infty \le \eps$.
    Then for every $g \in \Gcal$, there is an $i \in \{1, \dots, N\}$ such that
    \begin{align*}
        \sup_{S \in \Scal} \|g - f_i\|_{S} \le \sup_{S \in \Scal} \inf_{f \in \Fcal} \|f - g\|_{S} + \eps  \le 2\eps.
    \end{align*}
    This implies that $f_{1}, \dots, f_{N}$ are centers of an $\eps$-covering of $\Gcal$ with respect to the metric $\sup_{S \in \Scal} \|\cdot\|_{S}$. This is a contradiction since $N < |\Gcal|$ and all elements in $\Gcal$ are at least $2\eps$ apart in this metric.
\end{proof}

We proceed to the main proof.

\begin{enumerate}[(i)]
    \item We construct a set $\Gcal \subset \Hcal_s[0, 1]^D, C)$  satisfying
          \begin{align*}
              \inf_{f, g \in \Gcal, f \neq g} \sup_{\dim_M(\Mcal) \le d}\| f- g\|_{\Mcal} \gtrsim \eps, \quad \log |\Gcal| \gtrsim \eps^{-d/s} \log(\eps^{-1}).
          \end{align*}
          The lower bound then follows from \Cref{lem:covering-lower-bound} above.

          For $a \in (0, \infty)$, define the polynomial bump $\phi\colon [0, 1]^D \to \R$ as
          \begin{align*}
              \phi(x) = a\prod_{j=1}^D   x_j^{s+1}(1-x_j)^{s+1}.
          \end{align*}
          By straightforward computations, we can verify that $\phi\in \Hcal_s([0, 1]^D, aC')$ for some $C' < \infty$, and all derivatives up to order $\lfloor  s \rfloor$ vanish on the faces of $[0, 1]^D$.
          Next, let $K = \lceil \eps^{-1/s}\rceil$ and define half-open cubes
          \begin{align*}
              Q_{\bbeta} = \prod_{i=1}^D \biggl[\frac{\beta_i}{K}, \frac{\beta_i + 1}{K}\biggr), \qquad \bbeta \in \{0, \dots, K-1\}^D.
          \end{align*}
          For any $\bbeta$, define
          \begin{align*}
              \phi_{\bbeta, K}(\bx) = K^{-s}\phi(K(\bx - \bbeta/K)), \qquad \bx \in [0, 1]^D.
          \end{align*}
          Then for every multi-index $\balpha$ with $|\balpha| \le \lfloor s \rfloor $,
          \begin{align*}
              \| \partial^{\balpha} \phi_{\bbeta, K} \|_\infty \asymp aK^{|\balpha| - s} \le a,
          \end{align*}
          and for $|\balpha| = \lfloor s \rfloor$ and $\bx, \by \in [0, 1]^D$,
          \begin{align*}
              |\partial^{\balpha} \phi_{\bbeta, K}(\bx) - \partial^{\balpha} \phi_{\bbeta, K}(\by)|
               & \lesssim a K^{\lfloor s \rfloor - s} \|K(\bx - \by)\|_\infty^{s - \lfloor s \rfloor} \le a \|\bx - \by\|_\infty^{s - \lfloor s \rfloor}.
          \end{align*}

          For any function $\sigma\colon \{0, \dots, K-1\}^D \to \{0, 1\}$, define
          \begin{align*}
              f_{\bm \sigma}(\bx) = \sum_{\bbeta \in \{0, \dots, K-1\}^D} \sigma(\bbeta) \phi_{\bbeta, K}(\bx), \qquad \bx \in [0, 1]^D.
          \end{align*}
          Further, for any $\Bcal \subset \{0, \dots, K - 1\}^D$, define
          \begin{align*}
              \Sigma_{\Bcal} = \left\{\sigma\colon \{0, \dots, K-1\}^D \to \{0, 1\} \colon \sigma(\bbeta) = 0 \text{ for } \bbeta \notin \Bcal\right\}, \quad \Sigma = \bigcup_{|\Bcal| \le \eps^{-d/s}} \Sigma_{\Bcal}.
          \end{align*}
          Choosing $a$ small enough, we have  $\Gcal = \{f_{\bm \sigma} \colon \sigma \in \Sigma\} \subset \Hcal_s([0, 1]^D, C)$.

          for any $\sigma \neq \sigma' \in \Sigma$, it holds
          that
          \begin{align*}
              \sup_{\dim_M(\Mcal) \le d} \| f-g\|_{\Mcal} \ge K^{-s} \|\phi\|_\infty \gtrsim \eps.
          \end{align*}
          Finally, by standard combinatorial arguments,
          \begin{align*}
              \log |\Gcal| = \log |\Sigma| \ge \log \binom{K^D}{\lfloor \eps^{-d/s} \rfloor} \gtrsim \eps^{-d/s} \log(\eps^{-1}).
          \end{align*}

    \item Cover $\Mcal$ by $N = \Ncal(\eps^{1/s}, \Mcal, \|\cdot\|_\infty) \lesssim \eps^{-d / s}$ hypercubes $Q_1, \dots, Q_N$ of side-length $\eps^{1/s}$. For each hypercube $Q_k$, let $\bx_k$ be its center and define the function class
          \begin{align*}
              \Fcal_\Mcal = \left\{ f \colon f(\bx) = \sum_{k = 1}^N \ind_{Q_k}(\bx) \sum_{|\balpha| \le \lfloor s \rfloor } c_{k, \balpha} (\bx - \bx_k)^{\balpha} , c_{k, \balpha} \in \R, |c_{k, \balpha}| \le C \right\}.
          \end{align*}
          Then, by standard Taylor approximation arguments, for any $g \in \Hcal_s([0, 1]^D, C)$ there is an $f \in \Fcal_\Mcal$ such that
          \begin{align*}
              \sup_{\bx \in \Mcal} |g(\bx) - f(\bx)| \lesssim \eps.
          \end{align*}
          Moreover, the number of parameters in $\Fcal_\Mcal$ is at most $N \cdot \binom{D + \lfloor s \rfloor}{\lfloor s \rfloor} \lesssim \eps^{-d/s}$, which implies the desired bound on the covering number by Theorem 2.7.17 of \citet{van2023weak}.
\end{enumerate}

\subsection{Proof of \Cref{prop:optimality-bounded}}

Theorem 2.1  \citet{ou2024covering} implies that the covering number of 
$\Fcal_{D, 1}(N, L, B)$ satisfies
\begin{align*}
       \log \Ncal(\eps, \Fcal_{D, 1}(N, L, B), \|\cdot\|_\infty)     \lesssim N^2 L\log((N + 1)^LB^L / \eps).
\end{align*}
Therefore,
\begin{align*}
     & \quad \; \log \Ncal(\eps, \Fcal_{D, 1}(N, L + L \sqrt{\log L / \log N}, N^c), \|\cdot\|_\infty)            \\
     & \lesssim N^2 L^2 (1 + \log L / \log N) \log (N) +  N^2 L (1 + \sqrt{\log L / \log N}) \log(\eps^{-1}) \\
     & \lesssim N^2 L^2 \log (NL) +  N^2 L^2 \log(\eps^{-1}).
\end{align*}
Combined with \Cref{prop:optimality}, we have
\begin{align*}
    N^2 L^2 \log (NL) +  N^2 L^2 \log(\eps^{-1}) \gtrsim \eps^{-d/s} \log(\eps^{-1}),
\end{align*}
which implies the claim.

\subsection{Proof of \Cref{prop:point-fitting-ours}}

    By Lemma 13 of \citet{park2021provable} there is a unit vector $\bu \in \R^D$ such that for all $i \neq j$,
    \begin{align*}
        \sqrt{\frac{8}{\pi D}} \frac{1}{J^2} \| \bx_i - \bx_j\| \le |\bu^\top (\bx_i - \bx_j) | \le \| \bx_i - \bx_j\|.
    \end{align*}
    Defining $\wt \bu = \bu / \sqrt{D}$, we have
    \begin{align*}
        |\wt \bu^\top (\bx_i - \bx_j) | \ge  \sqrt{\frac{8}{\pi D^2}} \frac{1}{J^2} \ge R^{-1} \qquad \text{and} \qquad |\wt \bu^\top \bx_i| \le  1.
    \end{align*}
    Define $\phi_z \in \wt \Fcal_{D, 1}(1, 1, 1)$ by
    \begin{align*}
        \phi_z(\bx) = [1 + \sigma(\wt \bu^\top \bx) + \sigma(-\wt \bu^\top \bx)]/2 = (1 + \wt \bu^\top \bx)/2 \in [0, 1].
    \end{align*}
    Define $z_1 =  \phi_z(\bx_1), \dots, z_J =  \phi_z(\bx_J)$. We may assume without loss of generality that they are in ascending order, otherwise reorder them.
    Let $s = \log(R) + 1$.
    Since $| z_i - z_j| \ge R^{-1}/2 \ge 2^{-s}$ for $i \neq j$,  we can apply \Cref{lem:point-fitting-ours} below to construct a network $\phi_m \in \wt \Fcal_{1, 1}(N,  L + L [\sqrt{\log L} +  (s + r)  / \sqrt{\log (NL)}] / \sqrt{\log N}, N + 2^{(r + \log(R))/L})$ such that
    \begin{align*}
        \phi_m(z_j) = y_j \quad \text{for } j = 1, \dots, J.
    \end{align*}
    The final network is $\phi = \phi_m \circ \phi_z.$
    Since $\log(R) \lesssim \log(NL) + \log(\delta^{-1})$, and
    \begin{align*}
        \frac{\log(NL)}{\sqrt{\log N \log (NL)}} = \frac{\sqrt{\log(NL)}}{\sqrt{\log N}}  \lesssim 1 + \frac{\sqrt{\log L}}{\sqrt{\log N}},
    \end{align*}
    the proof is complete.

\begin{lemma} \label{lem:point-fitting-ours}
    Let $N, L, s \in \N$ with and suppose we are given numbers $x_1< x_2 < \dots < x_{J} \in [0, 1)$ with and $y_1, \dots, y_J \in \{0, \dots, 2^{r} -1\}$, where $|x_i - x_j| \ge 2^{-s}$ and $J \le N^2 L^2 $. Then there is a neural network $$\phi \in \wt \Fcal_{1, 1}(N, L + L [\sqrt{\log L} +  (s + r)  / \sqrt{\log (NL)}] / \sqrt{\log N}, N + 2^{(s + r) / L})$$ such that
    \begin{align*}
        \phi(x_j) = y_j \quad \text{for } j = 1, \dots, J.
    \end{align*}
\end{lemma}

\begin{proof} Without loss of generality we assume $J = N^2 L^2$, otherwise add arbitrary pairs $(x_j, y_j)$ to the set. Further let $n = \lfloor \log_3 N \rfloor$ and assume for simplicity that $L' =  L  \sqrt{n /\log (NL) } $  and  $M =N^2 L  \sqrt{\log (NL) / n}  $ are integers.
    \begin{enumerate}
        \item  Observe that $ J= ML'$. Construct $M$ disjoint intervals $I_m = [x_{(m - 1)L' + 1}, x_{mL'}], m = 1, \dots, M$, each containing exactly $L'$ of the $x_j$. 
        Let $s' = s + 2$, define $\hat x_j = \sum_{i=1}^{s'} 2^{-i} \bit_i( x_j)$, and observe that $|\hat x_j - x_j| \le 2^{-s'}$ and, for $j \neq k$,
        \begin{align*}
            | x_j - \hat x_k| \ge |x_j - x_k| - |x_k - \hat x_k| \ge 2^{-s}  - 2^{-s'} > 2^{-(s + 1)}.
        \end{align*}
        
        \item Let $c = \max\{r, s'\}$. For $ m = 1, \dots, M,$ define the numbers 
        \begin{align*}
            u_m &= \sum_{j=1}^{L'} \sum_{i=1}^{c} \bit_{i}(\hat x_{(m-1)L' + j}) 3^{-(j - 1)c + i}, \\ 
            v_m &= \sum_{j=1}^{L'} \sum_{i=1}^{c} \bit_{i}(y_{(m-1)L' + j} 2^{-r}) 3^{-(j - 1)c + i}.
        \end{align*}
        The number $u_m$ encodes all the digits of $\hat x_{(m - 1)L' + 1},\dots, \hat x_{mL' }$ and the number $v_m$ encodes all the digits of $y_{(m - 1)L' + 1},\dots, y_{mL' }$.
        We apply \Cref{lem:pw-linear} with $\bar y = 1$, $\delta = 2^{-s}$ to construct  neural networks $\phi_{u}, \phi_{w} \in \wt \Fcal_{1, 1}(N, L  \sqrt{\log (NL) / n}, (2^{4s}M^6)^{1/L})$ such that
        \begin{align*}
                  \phi_{u}(x) = u_m, \quad \phi_{w}(x) = w_m \quad \text{for } x \in I_m.
        \end{align*}
        Choosing $L \ge 4s$ ensures that $(2^{4s}M^6)^{1/L} \le 2 N (NL)^{12/L}$. For $L \ge 12$, we have have $(NL)^{12/L} \le 12N$ and therefore $\phi_u, \phi_w \in \wt \Fcal_{1, 1}(N, L  \sqrt{\log (NL) / n}, N)$.

        \item We next construct a decoder network that decodes $y_j$ from $x_j, u_m, w_m$ for $x_j \in I_m$.
        This networks implements an iterative procedure, decoding $\hat x_j$ and $y_j$ from $u_m$, checking whether $\hat x_j$ matches $x_j$, and outputting $y_j$ if they match, and $0$ otherwise.

        We illustrate the main mechanism for the first step of the iteration.
        Apply \Cref{lem:bit-decode} to construct a network $   \psi = (\psi_{1}, \psi_{2}) \in \wt \Fcal_{1, 2}(3^n, \lceil c / n \rceil, 3^n)$ such that
        \begin{align*}
              \psi(x) =  \left(\sum_{i = 1}^{c} \bit_{i}(x)2^{-i}, \sum_{i = 1}^{\infty} \bit_{c + i}(x)3^{-i}\right).
        \end{align*}
        Denote $\psi_{2}^{(k)}$ the $k$-th iteration of applying $\psi_{2}$, i.e.,   
        \begin{align*}
            \psi_{2}^{(1)}(x) & = \psi_{2}(x), \quad \psi_{2}^{(k)}(x) = \psi_{2}(\psi_{2}^{(k - 1)}(x)), \quad k = 2, 3, \dots,
        \end{align*}
        and observe that for $ k = 1, \dots, L'$, it holds that
        \begin{align*}
            \psi_{1} \circ \psi_{2}^{(k - 1)}(u_m) &= \hat x_{(m - 1)L' + k}, \qquad
            \psi_{1} \circ \psi_{2}^{(k - 1)}(w_m) =  2^{-r}y_{(m - 1)L' + k}.
        \end{align*}
        
        Next define a piecewise linear function $\phi_{\ind}(x)$ interpolating the points 
        $$(-2, 0), (-2^{-(s+1)}, 0), (-2^{-(s+2)}, 2^{-s}), (2^{-(s+2)}, 2^{-s}), (2^{-(s+1)}, 0), (2, 0).$$
        Then for any $j, k$,  it holds
        \begin{align*}
            \phi_{\ind}(x_j - \hat x_k) =  2^{-s}\ind_{j = k},
        \end{align*}
        and, by \Cref{lem:pw-linear}, we can realize this function by $\phi_{\ind} \in \wt \Fcal_{1, 1}(1, 1, 1)$.
        Define $\xi \in \wt \Fcal_{3, 1}(1, 1, 1)$ by
        \begin{align*}
            \xi(x, \hat x, y) = \sigma(y + \phi_{\ind}(x - \hat x) - 2^{-s}).
        \end{align*}
        Then for any $y \in [0, 2^{-s}]$ and $1 \le j, k \le J$, it holds that
        \begin{align*}
            \xi(x_j, \hat x_k, y) = \begin{cases}
                y \quad \text{if } j = k, \\
                0 \quad \text{otherwise}.
            \end{cases}
        \end{align*}
        In particular,
        \begin{align*}
            \xi\left(x_j, \psi_{1} \circ \psi_{2}^{(k - 1)}(u_m), 2^{-s}\psi_{1} \circ \psi_{2}^{(k - 1)}(w_m)\right) = \begin{cases}
                2^{-(s+r)} y_{j} \quad \text{if } j = (m - 1)L' + k, \\
                0 \quad \text{otherwise}.
            \end{cases}
        \end{align*}

        \item Equipped with these building blocks, we can now the final decoder network. We show by induction over $k = 1, \dots, L'$ that there is $\phi' \in \wt \Fcal_{4, 4}( 3^n, k\lceil s' / n \rceil,  3^n)$
        computing
        \begin{align} \label{eq:decoder-induction}
            \phi'(x_j, u_m, w_m, y) = 
            \begin{pmatrix}
                x_j \\
                \psi_{2}^{(k)}(u_m) \\
                \psi_{2}^{(k)}(w_m) \\
                2^{-(s+r)}y_j \ind_{(m - 1)L' \le j < (m - 1)L' + k}.
            \end{pmatrix}
        \end{align}
        The base case is straightforward: the first coordinate is constant, the second and third coordinates are computed by $\psi \in \wt \Fcal_{1, 2}(3^n, \lceil c / n \rceil, 3^n)$, and the last coordinate is 0. Now suppose the claim holds for $k - 1$, and let $\phi'' \in \wt \Fcal_{4, 4}( 3^n, (k - 1)\lceil c / n \rceil,  3^n)$ be the corresponding network. 
        Define
        \begin{align*}
            \phi'(x_j, u, w, y) = 
            \begin{pmatrix}
                x_j \\
                \psi_{2}(\phi''(x_j, u, w, y)_2) \\
                \psi_{2}(\phi''(x_j, u, w, y)_3) \\
                y + \xi\left(x_j, \psi_{1}(u), \psi_{1}(w)\right)
         \end{pmatrix}.
        \end{align*}
        The function can be implemented by a network $\phi' \in \wt \Fcal( 3^n, k\lceil c / n \rceil,  3^n)$ by stacking and chaining the networks for $\phi''$ and $\psi$ and $\xi$ appropriately.
        Further it satisfies the desired property \eqref{eq:decoder-induction}, completing the induction.

        \item Finally, let $\rho(x) = 2^{(s + r) / L} \, x$ and define $\rho^{(L)}$ as its $L$-fold iteration. Then $\rho^{(L)}(x) = 2^{(s + r)} x$ and $\rho^{(L)} \in \wt \Fcal_{1,1}(1, L, 2^{(s + r) / L})$.
         The final network is 
        \begin{align*}
            \phi(x) = \rho^{(L)}\left(\phi'\left(x, \phi_u(x), \phi_w(x), 0\right)_4\right).
        \end{align*}
         The stated network size follows upon noting that $3^n \le N$, 
        \begin{align*}
            L' \lceil c / n \rceil = L  \sqrt{n /\log (NL) }  \lceil c / n \rceil \lesssim L (s + r) / \sqrt{\log N \log (NL)},
        \end{align*}
        and
              \begin{align*}
                  \sqrt{\log(NL) / n} \le 1 + \sqrt{\log L / \log N}.
              \end{align*}
    \end{enumerate}
\end{proof}

\subsection{Proof of \Cref{thm:main-compositional}}

\begin{lemma}[Covering number on $\epsilon$-enlargement]
\label{cov}
    Let $A \subset R^m$, $r>0$ and $\eta >0$. Define the $\epsilon$-enlargement by 
    \begin{align}
    \label{enlarge}
        A^{+\epsilon} := \{x \in \R^m: \mathrm{dist}_{\infty}(x,A) \leq \epsilon\}.
    \end{align}
    Then
    \begin{align*}
        \mathcal{N}(\eta, A^{+\epsilon}, \|\cdot\|_{\infty}) \lesssim \mathcal{N}(\eta, A, \|\cdot\|_{\infty})\left(1+\frac{\epsilon}{\eta}\right)^m. 
    \end{align*}
    Consequently, if $\mathcal{N}(\eta, A, \|\cdot\|_{\infty}) \lesssim \eta^{-d}$ for small $\eta$, then
    \begin{align*}
        \mathcal{N}(\eta, A^{+\epsilon}, \|\cdot\|_{\infty}) \lesssim \eta^{-d} \left(1+\frac{\epsilon}{\eta}\right)^m.
    \end{align*}
\end{lemma}

\begin{proof}
    Let $\mathcal{N}(\eta, A, \|\cdot\|_{\infty}) =N$ and denote by $\{V_1, \dots, V_N\}$ the corresponding $\eta$-cover of $A$. Then
    \begin{align*}
        A^{+\epsilon} \subset \bigcup_{k=1}^N(V_k)^{+\epsilon}.
    \end{align*}
    Each $(V_k)^{+\epsilon}$ is contained in a translate of $[-\epsilon-\eta, \epsilon+\eta]^m$, which can be covered by $\lesssim ((\epsilon+\eta)/\eta)^m \lesssim (1+\epsilon/\eta)^m$. Multiplying by $N$ yields the claim.
\end{proof}

We proceed to the proof of the theorem.
Fix $\eps_0>0$ and set
\[
\eps_i := \eps_0 (C+1)^i,\qquad i=0,\dots,\ell,
\]
so that $\eps_{i+1}=(C+1)\eps_i$.
Assume $\eps_\ell\le 1$ and set $L'=L+L\sqrt{\log L/\log N}$.
Denote partial compositions
\[
G_i:=g_i \circ \cdots \circ g_0,
\qquad
\hat{G}_i := \hat{g}_i \circ \cdots \circ \hat{g}_0,
\]
for functios $\hat g_i$ to be chosen, $\phi = \hat G_\ell$,
and the cumulative errors 
\[
\delta_0 :=0, \qquad \delta_i := \sup_{\bx \in \mathcal{M}} \|G_i(\bx)-\hat{G}_i(\bx)\|_{\infty}, \quad i\ge 1.
\]

\medskip
\noindent\textbf{Step 1: construction on a fixed enlargement.}
For each level $i=0,\dots,\ell$ and each component $j=1,\dots,D_i$, we construct $\hat g_{ij}$ so that
\begin{equation}\label{eq:fixed-dom}
\sup_{\bz \in (\mathcal{M}_{ij})^{+\eps_i}} |g_{ij}(\bz)-\hat g_{ij}(\bz)| \le \eps_i.
\end{equation}
By \Cref{cov} with $\eta=\eps_i^{1/s_{ij}}$ and enlargement radius $\eps_i$,
\begin{align*}
\mathcal{N}\!\left(\eps_i^{1/s_{ij}}, (\mathcal{M}_{ij})^{+\eps_i}, \|\cdot\|_\infty\right)
&\lesssim
\mathcal{N}\!\left(\eps_i^{1/s_{ij}}, \mathcal{M}_{ij}, \|\cdot\|_\infty\right)
\left(1+\frac{\eps_i}{\eps_i^{1/s_{ij}}}\right)^{d_{ij}} \\
&\lesssim
\eps_i^{-d_{ij}/s_{ij}}
\left(1+\eps_i^{1-1/s_{ij}}\right)^{d_{ij}}
\lesssim
\eps_i^{-d_{ij}/s_{ij}},
\end{align*}
where in the last step we used $s_{ij}\ge 1$ and $\eps_i\le 1$ so that
$\eps_i^{1-1/s_{ij}}\le 1$.

Therefore, by \Cref{thm:main-approx}, there exist networks
$\hat g_{ij}\in \tilde \Fcal_{D_{i-1},1}(N,L',N)$ satisfying \eqref{eq:fixed-dom}, provided
\[
N^2L^2 \asymp \eps_i^{-d_{ij}/s_{ij}}
\quad\text{for all } i,j.
\]
In particular, it suffices to choose
\[
N^2L^2 \asymp \eps_0^{-d^*/s^*},
\qquad\text{where}\qquad
\frac{d^*}{s^*}:=\max_{i,j}\frac{d_{ij}}{s_{ij}}.
\]
Define $\hat g_i:=(\hat g_{i1},\dots,\hat g_{iD_i})$ and $\hat G_i:=\hat g_i\circ\cdots\circ\hat g_0$.

\medskip
\noindent\textbf{Step 2: Induction for the propagated error.}
We claim that for all $i=0,\dots,\ell$,
\begin{equation}\label{eq:delta-bound}
\delta_i \le \eps_{i+1}
\end{equation}
and $\hat{G}_i \in \tilde \Fcal_{D,D_i}(\max_{k}D_kN,iL',N)$. \\

\noindent\emph{Base case $i=1$.}
Since $\delta_0=0$, we have $(\mathcal{M}_{1j})^{+\delta_0}=\mathcal{M}_{1j}\subset (\mathcal{M}_{1j})^{+\eps_1}$.
Thus \eqref{eq:fixed-dom} yields
\[
\|G_1(\bx)-\hat{G}_1(\bx)\|_{\infty} = \sup_{\bz\in(\mathcal M_{1j})^{+\delta_0}} |g_{1j}(\bz)-\hat g_{1j}(\bz)|
\le
\sup_{\bz\in(\mathcal M_{1j})^{+\eps_1}} |g_{1j}(\bz)-\hat g_{1j}(\bz)|
\le \eps_1.
\]
Applying \Cref{lemprop} gives 
$\delta_1\le C\delta_0+\eps_{1}=\eps_1\le \eps_2$ and using parallelizing of the networks $\hat{g}_{ij}$, we have $\hat{G}_1=(\hat{g}_{11}, \dots, \hat{g}_{1D_1}) \in \tilde \Fcal_{D,D_1}(D_1N,L',N)$.\\

\noindent\emph{Inductive step.}
Assume $\delta_{i-1}\le \eps_i$. Then for any $\bx\in\mathcal M$ and any $j$,
with $\bu:=\pi_{S_{ij}}(G_{i-1}(\bx))\in\mathcal M_{ij}$ and
$\hat\bu:=\pi_{S_{ij}}(\hat G_{i-1}(\bx))$, we have
$\|\bu-\hat\bu\|_\infty \le \delta_{i-1}\le \eps_i$ and hence
\[
\hat\bu\in (\mathcal M_{ij})^{+\delta_{i-1}}
\subset (\mathcal M_{ij})^{+\eps_i}.
\]
Therefore \eqref{eq:fixed-dom} applies at $\hat\bu$ and gives
$|g_{ij}(\hat\bu)-\hat g_{ij}(\hat\bu)|\le \eps_i$, where $\hat{g}_{ij} \in \mathcal{F}_{D_{i-1},1}(N, L',N)$.
Using \Cref{lemprop} we conclude
\[
\delta_i \le C\delta_{i-1}+\eps_i \le C\eps_i+\eps_i = \eps_{i+1}
\]
and
\begin{align*}
    \delta_i = \sup_{\bx \in \Mcal}\|G_i(\bx) - \hat{G}_i(\bx)\|_\infty  \leq \eps_{i+1},
\end{align*}
where $\hat{G}_i=(\hat{g}_{i1}, \dots, \hat{g}_{iD_i})(G_{i-1}(\bx)) \in \mathcal{F}_{D,D_i}(\max_k D_k N, iL',N)$.
This proves the assertion.

\medskip
\noindent \textbf{Step 3: Choice of $\eps_0$.}
Finally, unrolling the recursion (or applying \Cref{lemprop}) yields
\[
 \sup_{\bx \in \Mcal} |G_\ell(\bx)-\hat G_\ell(\bx)| =\delta_\ell
\le \sum_{i=1}^{\ell} C^{\ell-i}\eps_i
= \sum_{i=1}^{\ell} C^{\ell-i}(C + 1)^i \eps_0
\le \ell (C + 1)^{\ell} \eps_0.
\]
With the choice
\[
\eps_{0}:= \frac{\eps}{\ell(C + 1)^\ell},
\]
we obtain $\delta_\ell\le \eps$ and $\hat{G}_{\ell} \in \mathcal{F}_{D,1}(\max_i D_i N, \ell L', N)$, completing the proof.

\section{Auxiliary results}

\subsection{Step-function approximation}

\begin{lemma} \label{lem:stepfun-2}
    For any \( N, L, r \in \mathbb{N}^{+} \) and $K = N^2L^2$,
    there exists a one-dimensional function \( \phi \in \wt \Fcal_{1,1}(\sqrt{N}, rL, N) \) such that
    \[
        \phi(x) = \frac{k}{K}, \quad \text{if } x \in \left[ \frac{k}{K}, \frac{k+1}{K} - \frac{\ind_{k \le K - 2}}{4 K^r}  \right], \quad \text{for } k = 0,1, \dots, K - 1.
    \]
\end{lemma}
\begin{proof}
    Let $K' = NL$ and observe that a piecewise linear function 
    satisfying $\phi(x) = 0$ for $x \le 0$, $\phi(x) = k/K'$ for $x \ge (K' - 1)/K'$, and
    otherwise interpolating the points
    \begin{align*}
       \bigcup_{k = 0}^{K' - 1} \left\{\left( \frac{k}{K'}, \frac{k}{K'} \right) \right\} \cup  \bigcup_{k = 0}^{K' - 2}  \left\{ \left( \frac{k+1}{K'} - \frac{1}{4 K^s}, \frac{k}{K'} \right) \right\}.
    \end{align*}
    Applying \Cref{lem:pw-linear} with $\bar y = 1$, $\delta = 1/(4 K^s)$, this
    can be realized as $\psi \in \wt \Fcal_{1,1}(\sqrt{N}, L, (K')^{6/L} K^{4s/L}) \subseteq \wt \Fcal_{1,1}(\sqrt{N}, sL, N)$.
    Based on this, we can construct a network $\phi \in  \wt \Fcal_{1,1}(\sqrt{N}, sL, N)$ computing
    \begin{align*}
        \phi(x) = \psi(x) + \psi(K'(x - \psi(x))) / K'.
    \end{align*}
    Fix $k \in \{0, 1, \dots, K - 1\}$ and let $x \in [k/K, (k+1)/K - 1/(4 K^s)]$.
    Define $k' := \lfloor K'x \rfloor $ and observe that $\lfloor Kx \rfloor = k$.
    Then it holds
    \begin{align*}
        \phi(x) = \frac{k'}{K'} + \frac{\lfloor K x - K ' k'  \rfloor}{K} = \frac{K'k'}{K} + \frac{\lfloor K x \rfloor - K ' k' }{K} = \frac{k}{K},
    \end{align*}
    as desired.
\end{proof}

\subsection{Median smoothing}

Define the modulus of continuity of a function $f$ as
\[
    \omega_f(\delta) = \sup_{\bx, \by \in [0,1]^D\colon \|\bx - \by\|_{\infty} \le \delta} |f(\bx) - f(\by)|, \quad \delta > 0.
\]
The following is a direct consequence of Lemma 3.4 of \citet{Lu2021} (restated in \Cref{lem:median-smoothing}) upon reparametrization.

\begin{lemma} \label{cor:median-smoothing}
    Let $\eps>0$, $M, K\in\mathbb{N}^+$, $\delta\in(0,1/(3K)]$,
    $\Bcal \subseteq \{ 0, \dots, K - 1\}^D$, and
    \begin{align*}
        Q_{\bbeta} = \left[ \frac{\beta_i}{K}, \frac{\beta_{i} +1}{K}\right]_{i = 1}^D, \quad \bbeta \in \Bcal.
    \end{align*}
    Assume $f\in C([0,1]^D)$ and $g:\mathbb{R}\to\mathbb{R}$ is a general function with
    \begin{align*}
        |g(\bx)-f(\bx)|\le \eps, \qquad \text{for any }  \bx \in \bigcup_{\bbeta \in \Bcal}  Q_{\bbeta} \setminus \Omega_{K, \delta},
    \end{align*}
    with $\Omega_{K, \delta}$ and $M_g$ defined as in \Cref{lem:median-smoothing}.
    Then
    \[
        |M_g(\bx)-f(\bx)|\le \eps + D\omega_f(\delta)
        \quad \text{for any }   \bx \in \bigcup_{\bbeta \in \Bcal} Q_{\bbeta}.
    \]
\end{lemma}
\begin{proof}
    Fix $\bbeta \in \Bcal$ and reparametrize
    \begin{align*}
        \tilde f(\bx) =  f(\bbeta/K + \bx/K), \; \tilde g(\bx) = g(\bbeta/K + \bx/K), \; M_{\tilde f}(\bx ) = M_{f}(\bx/K ), \; \tilde K = 1, \; \tilde \delta = \delta  K.
    \end{align*}
    Observe that $\tilde \delta \in (0, 1/(3\tilde K)]$, and
    \begin{align*}
        \omega_{\tilde f}(\tilde \delta) & = \sup_{\bx, \by \in [0,1]^D\colon \|\bx - \by\|_{\infty} \le \tilde \delta} |\tilde f(\bx) - \tilde f(\by)|,  \\
                                         & =  \sup_{\bx, \by \in [0,1]^D\colon \|\bx - \by\|_{\infty} \le \delta  K} | f(k/K +  \bx/K) - f(k/K + \by/K)|, \\
                                         & \le  \sup_{\bx, \by \in [0,1]^D\colon \|\bx - \by\|_{\infty} \le \delta} | f( \bx) - f(\by)|                   \\
                                         & = \omega_f(\delta).
    \end{align*}
    Then  \Cref{lem:median-smoothing} implies that
    \begin{align*}
        \sup_{\bx \in Q_{\bbeta}}|M_{\tilde g}(\bx)-\tilde f(\bx)|\le \eps + d\omega_{\tilde f}(\tilde \delta) \le  \eps + D\omega_f(\delta).
    \end{align*}
    Since $\bbeta \in \Bcal$ was arbitrary, the proof is complete.
\end{proof}

The next result follows from iterative application of \Cref{lem:mid-1d}.
\begin{lemma} \label{lem:mid}
    Suppose $N,L, B \in \N$, $\phi \in \Fcal_{D,1}(N, L, B)$, and define $M_\phi$ as in \Cref{lem:median-smoothing}. Then $M_\phi \in \Fcal_{D,1}(3^D(N + 4), L + 2D, B)$.
\end{lemma}
\begin{proof}
    We follow the proof Theorem 2.1 of \citet{Lu2021}.
    Implement the function
    \begin{align*}
        \Phi_{1}(\bx) = (\phi(\bx - \delta \be_1), \phi(\bx), \phi(\bx + \delta \be_1))^\top,
    \end{align*}
    by a ReLU network with width $3N$ and depth $L$ and weights bounded by $B$. By \Cref{lem:mid-1d}, the function $M_{\phi, 1} = \operatorname{mid} \circ \, \Phi_1$ can be implemented by a ReLU network with width $\max\{3N, 14\} \le 3(N + 4)$, depth $L + 2$, and weights bounded by $\max(B, 1) = B$.
    Repeating this procedure $D$ times, we obtain that $M_\phi = M_{\phi, D}$ can be implemented by a ReLU network with width $3^D(N + 4)$, depth $L + 2D$, and weights bounded by $B$.
\end{proof}

Together, we have the following proposition.
\begin{proposition} \label{prop:median-smoothing-nn}
    Let $\eps>0$, $M, K\in\mathbb{N}^+$, $\delta\in(0,1/(3K)]$,
    $\Bcal \subseteq \{ 0, \dots, K - 1\}^D$, and
    \begin{align*}
        Q_{\bbeta} = \left[ \frac{\beta_i}{K}, \frac{\beta_{i} +1}{K}\right]_{i = 1}^D, \quad \bbeta \in \Bcal.
    \end{align*}
    Assume $f\in C([0,1]^D)$ and $g \in \Fcal_{D,1}(N, L, B)$ are  such that
    \begin{align*}
        |g(\bx)-f(\bx)|\le \eps, \qquad \text{for any }  \bx \in \bigcup_{\bbeta \in \Bcal}  Q_{\bbeta} \setminus \Omega_{K, \delta},
    \end{align*}
    with $\Omega_{K, \delta}$ defined as in \Cref{lem:median-smoothing}.
    Then there  $\phi \in \Fcal_{D,1}(3^D(N + 4), L + 2D, B)$  such that
    \[
        |\phi(\bx)-f(\bx)|\le \eps + D\omega_f(\delta)
        \qquad \text{for any }   \bx \in \bigcup_{\bbeta \in \Bcal} Q_{\bbeta}.
    \]
\end{proposition}

\subsection{Bit decoding}

\begin{lemma} \label{lem:bit-decode}
        Let $n, \ell \in \mathbb{N}^+$.
    There exists a function
    \[
        \phi \in \wt \Fcal_{1,2}(3^{n}, \lceil \ell/n\rceil, 3^n)
    \]
    such that given any $\theta_1, \theta_2, \ldots \in \{0,1\}$, and $i \in \N$, it holds that
    \[
        \phi\left(\sum_{k = 1}^\infty \theta_k 3^{-k}\right)
        \;=\; \left(\sum_{j=1}^{\ell}  \theta_j 2^{-j}, \sum_{k = 1}^\infty \theta_{\ell + k} 3^{-k}  \right).
    \]
\end{lemma}
\begin{proof}
    Let $M = \lceil \ell/n \rceil$ and write 
    \begin{align*}
        \sum_{j=1}^{\ell}  \theta_j 2^{-j}
        = \sum_{m=1}^{M}  \left( \sum_{j=1}^{n} \theta_{(m-1) n + j} 2^{-(m-1) n - j} \ind_{(m -1) n + j \le \ell}  \right).
    \end{align*}
    Let $\psi$ be the function in \Cref{lem:bit-sum} and write it as $\psi(x, i) = (\psi(x, i)_1, \psi(x, i)_2)$. 
    For $m = 1, \dots, M$, define $\phi_m = (\phi_{m, 1}, \phi_{m, 2})$ as
    \begin{align*}
        \phi_{m, 1}(y, x) &=  y + \sum_{j=1}^{n} [\psi(x, j)_1 - \psi(x, j - 1)_1] 2^{-(m-1) n - j}  \ind_{(m -1) n + j \le \ell} , \\
         \phi_{m, 2}(y,x) &= \psi(x, \min\{n, \ell - (m-1) n\})_2.
    \end{align*}
    The function $\psi_j(x) = [\psi(x, j)_1 - \psi(x, j - 1)_1]$ can be implemented by parallelizing and subtracting two networks in $\wt \Fcal_{1,1}(2^{n}, 1, 3^n)$ by \Cref{lem:bit-sum}. The summation over $j$ can be implemented by parallelizing $n$ such networks and adding their outputs together with fixed weights, all of which bounded by 1. A final layer is required to add the result to $y$.  Thus, $\phi_m \in \wt \Fcal_{2,2}(n 2^{n}, 1, 3^n)$ for all $m = 1, \dots, M$. Since $n 2^{n} \le 2 \times  3^{n}$ for all $n \ge 1$, we can also write $\phi_m \in \wt \Fcal( 3^{n}, 1, 3^n)$.
    
    We prove the lemma by induction on $M$. For $M = 1$, define $\phi= \phi_{1}(0, \cdot) \in \wt \Fcal_{1,2}(3^{n}, 1, 3^n)$, which satisfies
    \begin{align*}
        \phi(x) &= \left( \sum_{j=1}^{n} \theta_j 2^{-j} \ind_{j \le \ell}, \sum_{k = 1}^\infty \theta_{n + k} 3^{-k}  \right),
    \end{align*}
    as required.
     Now suppose $M \ge 2$ and there is a function $\phi' \in \wt \Fcal_{1,2}(3^n, M - 1, 3^n)$ such that
    \begin{align*}
        \phi'\left(\sum_{k = 1}^\infty \theta_k 3^{-k}\right)
        \;=\; \left(\sum_{j=1}^{(M-1) n}  \theta_j 2^{-j}, \sum_{k = 1}^\infty \theta_{(M-1) n + k} 3^{-k}  \right).
    \end{align*}
    Then, the function $\phi = \phi_{M} \circ \phi' \in  \wt \Fcal_{1,2}(3^n, M, 3^n)$ satisfies
    \begin{align*}
        \phi'\left(\sum_{k = 1}^\infty \theta_k 3^{-k}\right)_1 &=  \phi_M\left(\sum_{j=1}^{(M-1) n}  \theta_j 2^{-j}, \sum_{k = 1}^\infty \theta_{(M-1) n + k} 3^{-k}\right)_1   \\
        &= \sum_{j=1}^{(M-1) n}  \theta_j 2^{-j} + \sum_{j=1}^{n} \theta_{(M-1) n + j} 2^{-(M-1) n - j} \ind_{(M -1) n + j \le \ell}  \\
        &= \sum_{j=1}^{\ell}  \theta_j 2^{-j} ,
    \end{align*}
    and
    \begin{align*}
        \phi'\left(\sum_{k = 1}^\infty \theta_k 3^{-k}\right)_2 &=  \phi_M\left(\sum_{j=1}^{(M-1) n}  \theta_j 2^{-j}, \sum_{k = 1}^\infty \theta_{(M-1) n + k} 3^{-k}\right)_2   \\
        &= \sum_{k = 1}^\infty \theta_{\min\{n, \ell - (M - 1)n\} + (M-1) n + k} 3^{-k}  \\
           &= \sum_{k = 1}^\infty \theta_{\ell + k} 3^{-k}  .
\end{align*}
This completes the induction and hence the proof.
\end{proof}

\section{Useful results from other papers}

\begin{lemma}[Proposition C.1 in \citet{ou2024three}] \label{lem:pw-linear}
    Let $M, N, L \in \N$ with $3 \le M \le N^2L$, and
    \( x_1 < x_2 < \dots < x_{M} \) and $y_1, \dots, y_M \in \R$. Denote $\delta = \min_{1 \le i \le M} | x_{i} - x_{i - 1}| / (\max\{1, \max_{1 \le i \le M} |x_i|\}$ and $\bar y = \max_{0 \le i \le M} |y_i|$. Then there is $\phi \in \wt \Fcal_{1,1}(N, L, (M^6 \delta^{-4} \bar y)^{1/L})$ such that
    \begin{enumerate}[(i)]
        \item \( \phi(x_i) = y_i \) for \( i = 1, \dots, M \).
        \item \( \phi \) is affine on each interval \( [x_{i-1}, x_i] \) for \( 2 \le i \le M \) and constant elsewhere.
    \end{enumerate}
\end{lemma}

\begin{lemma}[Lemma D.1 in \citet{ou2024three}, simplified]\label{lem:bit-sum}
    Let $n \in \mathbb{N}^+$.
    There exists a function
    \[
        \phi \in \wt \Fcal_{1,1}(2^{n}, 1, 3^n)
    \]
    such that given any $\theta_1, \theta_2, \ldots \in \{0,1\}$, and $i \in \N$, we have
    \[
        \phi\left(\sum_{k = 1}^\infty \theta_k 3^{-k},\, i\right)
        \;=\; \left(\sum_{j=1}^{\min\{i, n\}} \theta_j, \sum_{k = 1}^\infty \theta_{n + k} 3^{-k}  \right).
    \]
\end{lemma}

\begin{lemma}[Proposition 4.1 of \citet{Lu2021}]\label{lem:poly-approx}
Let 
\(\alpha\in\mathbb{N}^d\) with \(\lvert \alpha\rvert_1\le k\in\mathbb{N}^+\).
For any \(N,L\in\mathbb{N}^+\), there exists $\phi \in \Fcal_{1,1}(9(N+1)+k-1, 7k^2L, 2)$ such that
\begin{align*}
| \phi(x) - x_1^{\alpha_1}x_2^{\alpha_2}\cdots x_d^{\alpha_d} | \le 9k(N+1)^{-7kL}
\quad\text{for any } x\in[0,1]^d.
\end{align*}
\end{lemma}

\begin{lemma}[Lemma 3.1 of \citet{Lu2021}] \label{lem:mid-1d}
    The median function $\operatorname{mid}\colon\mathbb{R}^3\to\mathbb{R}$ satisfies 
    $\operatorname{mid} \in \Fcal_{3,1}(14, 2, 1)$.
\end{lemma}
\begin{remark}
    The weight bounds are not stated in the original versions of the two previous results, but can easily be traced back from their proofs.
\end{remark}

\begin{lemma}[Lemma 3.4 of \citet{Lu2021}] \label{lem:median-smoothing}
    Given any $\eps>0$, $K\in\mathbb{N}^+$, and $\delta\in(0,1/(3K)]$, define
    \begin{align*}
        \Omega_{K, \delta} = \bigcup_{j = 1}^D \left\{\bx \in [0,1]^D\colon x_j \in \bigcup_{k = 1}^{K - 1} \left(\frac{k}{K} - \delta,  \frac{k}{K}\right)\right\},
    \end{align*}
    and assume $f\in C([0,1]^D)$ and
    $g\colon\mathbb{R}^D\to\mathbb{R}$ is a general function with
    \[
        |g(\bx)-f(\bx)|\le \eps,\
        \quad \text{for any } \bx \in [0,1]^D\setminus\Omega([0,1]^D,K,\delta).
    \]
    Then
    \[
        |M_g(\bx)-f(\bx)|\le \eps + D\cdot \omega_f(\delta)
        \quad \text{for any } \bx \in [0,1]^D,
    \]
    where $M_g:=M_{g, D}$ is defined by induction through $M_{g,0}=g$ and
    \begin{equation}
        M_{g, i+1}(\bx)
        := \operatorname{mid}\bigl(M_{g,i}(\bx-\delta \be_{i+1}),\,
        M_{g,i}(\bx),\,M_{g,i}(\bx+\delta \be_{i+1})\bigr),
        \quad i=0,1,\ldots,D-1,
        \tag{3.4}
    \end{equation}
    where $\{\be_i\}_{i=1}^D$ is the standard basis in $\mathbb{R}^D$.
\end{lemma}

\end{document}